\PassOptionsToPackage{dvipsnames}{xcolor}
\pdfoutput=1
\documentclass{article}

\usepackage[utf8]{inputenc}
\usepackage{amssymb}
\usepackage{amsfonts}
\usepackage[tbtags]{amsmath}
\usepackage{amsthm}
\usepackage{romanbar}
\usepackage{tikz}
\usepackage{float}
\usepackage[left=3.5cm,right=3.5cm,top=3.5cm,bottom=3.5cm]{geometry}
\usetikzlibrary{cd}
\usetikzlibrary{decorations.pathmorphing}
\usetikzlibrary{positioning}
\usetikzlibrary{patterns}
\usetikzlibrary{patterns.meta}
\usepackage[backend=biber,style=alphabetic]{biblatex}
\usepackage{titlesec}
\usepackage[dvipsnames]{xcolor}
\usepackage{stmaryrd}
\usepackage{enumitem}
\usepackage{titletoc}
\usepackage{hyperref}

\newcommand\Hom[3][]{\mathrm{Hom}_{#1}(#2,#3)}
\newcommand\m[1]{\mathrm{mod}(#1)}

\newcommand\Db[1]{\mathrm{D}^{\mathrm{b}}(#1)}
\newcommand\per[1]{\mathrm{per}(#1)}

\newcommand\id[1]{\mathrm{id}_{#1}}

\newcommand\lotimes[2]{\overset{\mbox{\tiny \bf{L}}}{\otimes}_{#1}#2}

\let\ens\mathbb
\newcommand\Z{\ens{Z}}

\tikzstyle{etiquette}=[minimum height=0.8cm, minimum width=1cm]

\newtheoremstyle{definition}
{\topsep}
{\topsep}
{\normalfont}
{0pt}
{\bfseries\fontfamily{bch}\selectfont}
{.}
{ }
{\thmname{#1}\thmnumber{ #2}\textnormal{\thmnote{ (#3)}}}

\newtheoremstyle{notation}{\topsep}{\topsep}{\normalfont}{0pt}{\bfseries\fontfamily{bch}\selectfont}{.}{ }{\thmname{#1}\thmnumber{ #2}\textnormal{\thmnote{ (#3)}}}

\newtheoremstyle{lemma}{\topsep}{\topsep}{\itshape}{0pt}{\bfseries\fontfamily{bch}\selectfont}{.}{ }{\thmname{#1}\thmnumber{ #2}\textnormal{\thmnote{ (#3)}}}

\newtheoremstyle{proposition}{\topsep}{\topsep}{\itshape}{0pt}{\bfseries\fontfamily{bch}\selectfont}{.}{ }{\thmname{#1}\thmnumber{ #2}\textnormal{\thmnote{ (#3)}}}

\newtheoremstyle{corollary}{\topsep}{\topsep}{\itshape}{0pt}{\bfseries\fontfamily{bch}\selectfont}{.}{ }{\thmname{#1}\thmnumber{ #2}\textnormal{\thmnote{ (#3)}}}

\newtheoremstyle{theorem}{\topsep}{\topsep}{\itshape}{0pt}{\bfseries\fontfamily{bch}\selectfont}{.}{ }{\thmname{#1}\thmnumber{ #2}\textnormal{\thmnote{ (#3)}}}

\newtheoremstyle{example}{\topsep}{\topsep}{\normalfont}{0pt}{\bfseries\fontfamily{bch}\selectfont}{.}{ }{\thmname{#1}\thmnumber{ #2}\textnormal{\thmnote{ (#3)}}}

\newtheoremstyle{remark}{\topsep}{\topsep}{\normalfont}{0pt}{\bfseries\fontfamily{bch}\selectfont}{.}{ }{\thmname{#1}\thmnumber{ #2}\textnormal{\thmnote{ (#3)}}}

\theoremstyle{definition}
\newtheorem{defn}{Definition}[section]

\theoremstyle{notation}
\newtheorem*{nota}{Notation}

\theoremstyle{lemma}
\newtheorem{lem}[defn]{Lemma}

\theoremstyle{proposition}
\newtheorem{prop}[defn]{Proposition}

\theoremstyle{corollary}

\theoremstyle{theorem}
\newtheorem{thm}[defn]{Theorem}
\newtheorem*{thm*}{Theorem}

\theoremstyle{example}
\newtheorem{ex}[defn]{Example}

\theoremstyle{remark}
\newtheorem{rem}[defn]{Remark}

\titleformat{\section}[block]
{\centering\normalfont\Large\fontfamily{qhv}\selectfont}
{\thesection.}{0.5em}{}

\titleformat{\subsection}[block]
{\flushleft\normalfont\large\itshape\fontfamily{cmss}\selectfont}
{\thesubsection.}{0.5em}{}

\titlecontents{section}[0pt]{\bfseries\addvspace{0.1cm}}
{\thecontentslabel. \hspace{0.1mm}}
{}
{\hfill\contentspage}

\titlecontents{subsection}[15pt]{}
{\thecontentslabel. \hspace{0.1mm}}
{}
{\titlerule*[0.5pc]{.}\contentspage}

\title{\fontfamily{qhv}\selectfont Generalized Kauer moves and derived equivalences of Brauer graph algebras}
\author{\scshape Valentine Soto}
\date{}

\begin{document}
\maketitle

\medskip

\begin{center}
\begin{minipage}{0.75\textwidth}
\textbf{\fontfamily{bch}\selectfont Abstract.} Kauer moves are local moves of an edge in a Brauer graph that yield derived equivalences between Brauer graph algebras \cite{Kauer}. These derived equivalences may be interpreted in terms of silting mutations. In this paper, we generalize the notion of Kauer moves to any finite number of edges. Their construction is based on cutting and pasting actions on the Brauer graph. To define these actions, we use an alternative definition of Brauer graphs coming from combinatorial topology \cite{Lazarus}. Using the link between Brauer graph algebras and gentle algebras via the trivial extension \cite{Schroll2}, we show that the generalized Kauer moves also yield derived equivalences of Brauer graph algebras and also may be interpreted in terms of silting mutations.
\end{minipage}
\end{center}

\medskip

\tableofcontents

\medskip

\phantomsection
\section*{Introduction}
\addcontentsline{toc}{section}{Introduction}

\medskip

Brauer graph algebras are finite dimensional algebras defined thanks to a combinatorial data called Brauer graphs. They were first introduced in modular representation theory of finite groups by Donovan and Freislich \cite{DF}. Thereafter, these algebras appear in other contexts such as in connection with cluster theory \cite{MS,Ladkani}, with Hecke algebras \cite{Ariki, Ariki2}, with gentle algebras via the trivial extension \cite{Schroll2} and with dessins d'enfants \cite{MS2}. Moreover, Brauer graph algebras coincide with symmetric special biserial algebras \cite{Schroll2}, whose representation theory has been a lot studied. For instance, this correspondence has allowed to classify finite dimensional indecomposable modules over such algebras in terms of string and band combinatorics \cite{DF, WW} and to classify their Auslander-Reiten components \cite{ES}. We refer to \cite{Schroll} for a detailed list of results linked to Brauer graph algebras.

A Brauer graph is a finite graph equipped with a multiplicity function defined on each vertex and an orientation of the edges around each vertex. It can also be defined as a graph embedded into an oriented surface, together with a multiplicity function defined on each vertex \cite{MS}. This alternative definition of a Brauer graph has added a geometric point of view in the study of representation theory of Brauer graph algebras. One reason for the ongoing interest for this class of algebra is that the data of the graph and the surface associated to a Brauer graph algebra plays an important role in the representation theory of Brauer graph algebras. The notion of Brauer graph also appear in combinatorial topology (see \cite{Lazarus} for instance) and is called an oriented map in this context.

In recent years, there has been a particular interest in the study of derived equivalences of Brauer graph algebras. It has been proved that the class of Brauer graph algebras is stable under derived equivalences \cite{AZ}. Moreover, Opper and Zvonareva have classified Brauer graph algebras up to derived equivalences \cite{OZ}. This classification is based on the work of Antipov \cite{Antipov} and relies on $A_{\infty}$-categories associated to certain collections of arcs on graded punctured surface. In this paper, we are interested in a combinatorial tool, called Kauer moves, that yields derived equivalences of Brauer graph algebras with multiplicity identically one \cite{Kauer}. This Kauer move is a local move of an edge in the associated Brauer graph. These derived equivalences are given by Okuyama-Rickard tilting complexes \cite{Okuyama}, which can be interpreted in terms of silting mutations in the derived category as defined in \cite{AI}. 

The goal of this paper is to generalize the notion of Kauer moves to any finite number of edges. We want to extend this definition so that the generalized Kauer moves also yield derived equivalences of Brauer graph algebras and can also be interpreted in terms of silting mutations. Our main result is the following.

\medskip

\begin{thm*}[Theorem \ref{thm:generalized Kauer move}]
Let $B$ be a Brauer graph algebra with multiplicity identically one. We set $\Gamma$ to be the underlying Brauer graph of $B$ and $\Gamma_{1}$ to be the set of edges of $\Gamma$. Then for any subset $E$ of $\Gamma_{1}$,

\smallskip

\begin{enumerate}[label=(\arabic*)]
\item The Brauer graph algebra $B_{E}$ obtained by a generalized Kauer move of $E$ in $\Gamma$ is derived equivalent to $B$.
\item The silting mutation of $B$ over the projective $B$-module associated to the set of edges $\Gamma_{1}\backslash E$ is tilting. Its endomorphism algebra is isomorphic to $B_{E}$.
\end{enumerate}
\end{thm*}

\medskip

The key points to prove this theorem is the link between Brauer graph algebras and gentle algebras via the trivial extension \cite{Schroll2} and the fact that derived equivalences between finite dimensional algebras yield derived equivalences between their trivial extensions \cite{Rickard}. We use the topological model of the bounded derived category of gentle algebras introduced in \cite{OPS} to interpret silting mutation in gentle algebras in terms of moves of several arcs in its admissible dissection. We will prove in Example \ref{ex:succession of standard Kauer moves 2} that the generalized Kauer moves cannot be described as a succession of standard Kauer moves in general. Therefore, Theorem \ref{thm:mutation in gentle algebra} generalizes Corollary 3.7 in \cite{CS}.

In the first section, we define the notion of generalized Kauer moves. We first recall the definition and some results on Brauer graph algebras from \cite{Schroll}. To define these generalized Kauer moves, we will use an alternative definition of Brauer graphs arising from combinatorial topology. This allows us to define a cutting and pasting action on Brauer graphs.

In the second section, we adapt the notion of generalized Kauer moves for gentle algebras. In order to do so, we use the fact that gentle algebras are in bijection with Brauer graph algebras (of multiplicity identically one) equipped with an admissible cut \cite{Schroll}. Then, we use the correspondence between gentle algebras and marked surfaces equipped with an admissible dissection \cite{OPS} to interpret moves of admissible dissections in terms of the generalized Kauer moves for gentle algebras.

In the last section, we interpret our notion of generalized Kauer move for Brauer graph algebras and gentle algebras in terms of silting mutations. We begin the section with the definition and some results on silting and tilting mutations from \cite{AI}. Our main result will be proved in this section.

\medskip

\noindent \textbf{\fontfamily{bch}\selectfont Acknowledgments.} This paper is part of a PhD thesis supervised by Claire Amiot and supported by the CNRS. The author thanks her advisor for the useful advice and comments on the paper. Part of this work was done while the author was in Sherbrooke University. She also thanks Thomas Brüstle for helpful discussions and Sherbrooke University for warm welcome.

\medskip

\phantomsection
\section*{Notations and conventions}
\addcontentsline{toc}{section}{Notations and conventions}

\medskip

In this paper, all algebras are supposed to be over an algebraically closed field $k$. For any algebra $\Lambda$, we denote by $\m{\Lambda}$ the category of finitely generated right $\Lambda$-modules and  $\mathrm{proj}(\Lambda)$ the full subcategory of $\m{\Lambda}$ consisting of projective modules. Furthermore, $\Db{\Lambda}$ denotes the bounded derived category associated to $\m{\Lambda}$ and $\per{\Lambda}$ denotes the full subcategory of $\Db{\Lambda}$ consisting of perfect complexes. Moreover, we denote by $\mathrm{Triv}(\Lambda)$ the trivial extension of $\Lambda$ by the $\Lambda$-$\Lambda$-bimodule $D\Lambda=\Hom[k]{\Lambda}{k}$ : it is the algebra defined by $\mathrm{Triv}(\Lambda)=\Lambda\oplus D\Lambda$ as a $k$-vector space and whose multiplication is induced by the $\Lambda$-$\Lambda$-bimodule structure of $D\Lambda$ 

\[(a,\phi).(b,\psi)=(ab, a\psi+\phi b)\]

\noindent for all $a,b\in \Lambda$ and $\phi,\psi\in D\Lambda$. Finally, arrows in a quiver will be composed from right to left : for arrows $\alpha$, $\beta$, we write $\beta \alpha$ for the path from the start of $\alpha$ to the target of $\beta$. Similarly, morphisms will be composed from right to left.

\medskip

\section{Generalized Kauer moves}

\medskip

The goal of this section is to generalize the notion of Kauer moves. We extend it from one edge of a Brauer graph to any finite number of edges. These Kauer moves are local moves on the Brauer graph that yield derived equivalences between the corresponding Brauer graph algebras. We also want to have such result for the generalized Kauer moves. 

\medskip

\subsection{Brauer graph algebras}

\medskip

In this part, we recall the definition of a Brauer graph algebra and how they are closely linked to gentle algebras. We refer to \cite{Schroll} for more information on Brauer graph algebras.

\medskip

\begin{defn} \label{def:Brauer graph}
    A \textit{Brauer graph (of multiplicity 1)} is a tuple $\Gamma=(\Gamma_{0},\Gamma_{1},\circ)$ where 

    \smallskip

    \begin{itemize}[label=\textbullet, font=\tiny]
        \item $(\Gamma_{0},\Gamma_{1})$ is a finite graph whose vertex set is $\Gamma_{0}$ and edge set is $\Gamma_{1}$,
        \item $\circ$ is called the \textit{orientation} of $\Gamma$ and is given for every vertex $v\in \Gamma_{0}$ by a cyclic ordering of the edges incident with $v$. Note that the cyclic ordering of a vertex incident to a single edge $i$ is given by $i$.
    \end{itemize}
\end{defn}

\medskip

\begin{rem}
    In this paper, we are only interested in Brauer graphs of multiplicity 1. More generally, a Brauer graph also contains a multiplicity function $m:\Gamma_{0}\rightarrow\Z_{>0}$ \cite{Schroll}.
\end{rem}

\medskip

Unless otherwise stated, the orientation of a Brauer graph will be given by locally embedding each vertex into the counterclockwise oriented plane.

\medskip

\begin{ex} \label{ex:Brauer graph} Let $\Gamma$ be the Brauer graph given by

\smallskip

    
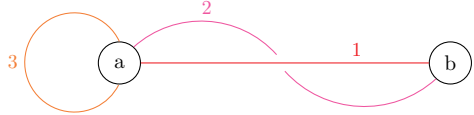
\begin{figure}[H]
        \centering
        \begin{tikzpicture}[scale=1.1]
        \tikzstyle{vertex}=[minimum size=0.7cm,circle,draw,scale=0.8]
        \node[vertex] (a) at (-2,0) {a};
        \node[vertex] (b) at (2,0) {b};
        \draw[Red] (a)--(b) node[near end, above, scale=0.75]{1};
        \draw[VioletRed] (a) to[bend left=45] node[midway, above, scale=0.75]{2} (-0.1,0.1); 
        \draw[VioletRed] (0,-0.1) to[bend right=45] (b) ;
        \draw[Orange] (a.90) arc(25:333:0.6) node[left, midway, scale=0.75]{3};
        \end{tikzpicture}
        \label{Brauer graph}
    \end{figure}

\smallskip

\noindent Using previous convention, the cyclic ordering at vertex $a$ is given by $1<2<3<3<1$ and at vertex $b$ by $1<2<1$.
\end{ex}

\medskip

Given a Brauer graph $\Gamma=(\Gamma_{0},\Gamma_{1},\circ)$, we construct a quiver $Q_{\Gamma}=(Q_{0},Q_{1})$ as follows \cite{Schroll}

\smallskip

\begin{itemize}[label=\textbullet, font=\tiny]
    \item The vertex set $Q_{0}$ of $Q_{\Gamma}$ is the edge set $\Gamma_{1}$ of the Brauer graph $\Gamma$,
    \item The arrow set $Q_{1}$ of $Q_{\Gamma}$ is induced by the orientation $\circ$ of $\Gamma$. More precisely, given two edges $i$ and $j$ of $\Gamma$ incident with the same vertex $v\in \Gamma_{0}$, there is an arrow in $Q_{\Gamma}$ from $i$ to $j$ if $j$ is the direct successor of $i$ in the cyclic ordering of $v$. 
\end{itemize}

\smallskip

\noindent Notice that we can associate an oriented cycle $C_{v}$ in $Q_{\Gamma}$ to any vertex $v\in\Gamma_{0}$ that is not incident to a single edge. This oriented cycle $C_{v}$ is unique up to cyclic permutation and is called a \textit{special cycle} at $v$. Moreover, we call a \textit{special $i$-cycle} at $v$ a representative of $C_{v}$ in its cyclic permutation class that begins and ends with $i\in Q_{0}$. Note that a special $i$-cycle at $v$ is not necessarily unique but there are at most two special $i$-cycles at $v$ and this happens exactly when $i$ is a loop at $v$.

\medskip

\begin{ex} \label{ex:quiver of a Brauer graph}
    Let $\Gamma$ be the Brauer graph defined in Example \ref{ex:Brauer graph}. By construction, its associated quiver is given by

    \smallskip

    
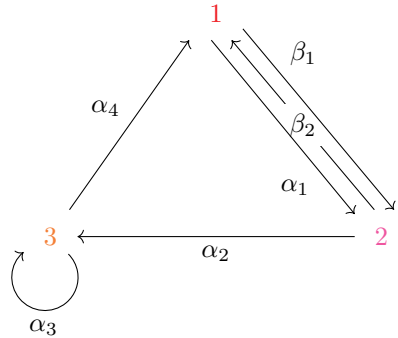
\begin{figure}[H]
        \centering
        \begin{tikzpicture}[scale=1.1]
        \node[minimum size=0.7cm, Red] (1) at (0,2.7) {1};
        \node[minimum size=0.7cm, VioletRed] (2) at (2,0) {2};
        \node[minimum size=0.7cm, Orange] (3) at (-2,0) {3};
        \draw[->] (2)--(3) node[midway, below]{$\alpha_{2}$};
        \draw[<-] (3.210) arc(130:405:0.4) node[midway, below]{$\alpha_{3}$};
        \draw[->] (3)--(1.-135) node[midway, above left]{$\alpha_{4}$};
        \draw[<-] (1.-60)--(2.105) node[fill=white, midway]{$\beta_{2}$};
        \draw[->] (1.-100)--(2.140) node[near end, below left]{$\alpha_{1}$};
        \draw[->] (1.-30)--(2.65) node[near start, above right]{$\beta_{1}$};
        \end{tikzpicture}
        \label{Quiver of a Brauer graph}
    \end{figure}

    \smallskip

    \noindent There are two special cycles in this quiver given by $C_{a}=\alpha_{4}\alpha_{3}\alpha_{2}\alpha_{1}$ and $C_{b}=\beta_{2}\beta_{1}$. Since 3 is a loop at $a$ in $\Gamma$, there are two distinct special 3-cycles at $a$ given by $\alpha_{2}\alpha_{1}\alpha_{4}\alpha_{3}$ and $\alpha_{3}\alpha_{2}\alpha_{1}\alpha_{4}$.
\end{ex}

\medskip

\begin{defn}\label{def:Brauer graph algebra}
    Let $\Gamma$ be a Brauer graph and $Q_{\Gamma}=(Q_{0},Q_{1})$ be its associated quiver. The \textit{Brauer graph algebra} $B_{\Gamma}$ associated to $\Gamma$ is the path algebra $kQ_{\Gamma}/I_{\Gamma}$ where the ideal of relations $I_{\Gamma}$ is generated by three types of relations given by

    \smallskip
    
    \begin{enumerate}[label=(\Roman*)]
        \item \label{item1:Brauer graph algebra}\[C_{v}^{i}-C_{v'}^{i}\]
        
        for any $i\in Q_{0}$ and for any special $i$-cycle $C_{v}^{i}$ and $C_{v'}^{i}$ at $v$ and $v'$ respectively, where $v$ and $v'$ are not incident to a single edge.

        \item \label{item2:Brauer graph algebra}\[\alpha_{n}C_{v}^{i}\]

        for any $i\in Q_{0}$ and for any special $i$-cycle $C_{v}^{i}=\alpha_{1}\ldots \alpha_{n}$ at $v$, where $v$ is not incident to a single edge.

        \item \label{item3:Brauer graph algebra}\[\beta\alpha\]

        for any $\alpha,\beta\in Q_{1}$ such that $\beta\alpha$ is not a subpath of any special cycles.
    \end{enumerate}
\end{defn}

\medskip

In general, these relations are not minimal as we can see in the following example.

\medskip

\begin{ex} \label{ex:relations of Brauer graph algebra}
    Let $\Gamma$ be the Brauer graph defined in Example \ref{ex:Brauer graph}. The set of relations of its associated Brauer graph algebra is given by

    \smallskip

    \begin{enumerate}[label=(\Roman*)]
    \item \label{I} $\alpha_{4}\alpha_{3}\alpha_{2}\alpha_{1}-\beta_{2}\beta_{1}$, $\alpha_{1}\alpha_{4}\alpha_{3}\alpha_{2}-\beta_{1}\beta_{2}$, $\alpha_{2}\alpha_{1}\alpha_{4}\alpha_{3}-\alpha_{3}\alpha_{2}\alpha_{1}\alpha_{4}$;
    \item \label{II} $\alpha_{1}\alpha_{4}\alpha_{3}\alpha_{2}\alpha_{1}$, $\beta_{1}\beta_{2}\beta_{1}$, $\alpha_{2}\alpha_{1}\alpha_{4}\alpha_{3}\alpha_{2}$, $\beta_{2}\beta_{1}\beta_{2}$, $\alpha_{3}\alpha_{2}\alpha_{1}\alpha_{4}\alpha_{3}$, $\alpha_{4}\alpha_{3}\alpha_{2}\alpha_{1}\alpha_{4}$;
    \item \label{III} $\beta_{1}\alpha_{4}$, $\alpha_{1}\beta_{2}$, $\alpha_{2}\beta_{1}$, $\beta_{2}\alpha_{1}$, $\alpha_{4}\alpha_{2}$, $\alpha_{3}^{2}$.
    \end{enumerate}

    \smallskip

    \noindent In this case the relations \ref{II} are determined by \ref{I} and \ref{III}.
\end{ex}

\medskip

Let us recall now the link between gentle algebras and Brauer graph algebras via the trivial extension. We refer for instance to \cite[Definition 3.1]{Schroll} for the definition of a gentle algebra.

\medskip

\begin{defn} \label{def:admissible cut}
    Let $Q$ be a quiver associated to a Brauer graph $\Gamma$. An \textit{admissible cut} $C$ of $Q$ is a set of arrows of $Q$ containing exactly one arrow for each special cycle $C_{v}$ (up to cyclic permutation) for any vertex $v$ in $\Gamma_{0}$ that is not incident to a single edge.
\end{defn}

\medskip 

\begin{thm}[Schroll {\cite[Theorem 3.13 and Corollary 3.14]{Schroll}}] \label{thm:Schroll}
    Let $B=kQ/I$ be a Brauer graph algebra of Brauer graph $\Gamma$ and $C$ be an admissible cut of $Q$. Then

    \smallskip

    \begin{itemize}[label=\textbullet, font=\tiny]
        \item The path algebra $B_{C}:=kQ/<I\cup C>$ is a gentle algebra, where $<I\cup C>$ is the ideal of $kQ$ generated by $I\cup C$.

        \item $B$ is the trivial extension of $B_{C}$.
    \end{itemize}

    \smallskip

    \noindent Conversely, for every gentle algebra $\Lambda$, 
    there exists a unique Brauer graph algebra $B=kQ/I$ such that $\Lambda=B_{C}$ for some admissible cut $C$ of $Q$. Moreover, this Brauer graph algebra is the trivial extension $\mathrm{Triv}(\Lambda)$ of $\Lambda$.
\end{thm}

\medskip

\subsection{Generalized Kauer moves}

\medskip

Let us recall that a Kauer move at an edge $s$ of a Brauer graph is a local move that can be described as follows \cite{Kauer}

\smallskip

\begin{figure}[H]
    \centering
        \begin{tikzpicture}[scale=0.8]
        \tikzstyle{vertex}=[draw, circle, minimum size=0.2cm]
        \begin{scope}[xshift=-7cm]
        \node[vertex] (1) at (-2,0) {};
        \node[vertex] (2) at (2,0) {};
        \node[vertex] (3) at (0,1.7) {};
        \node[vertex] (4) at (0,-1.7) {};
        \draw[very thick,Aquamarine] (1)--(2) node[above, midway]{$s$};
        \draw (1)--(3);
        \draw (2)--(4);
        \draw[->] (-1,0.1) arc(0:30:1);
        \draw[->] (1,-0.1) arc(180:210:1);
        \end{scope}
        \draw[->] (-4,0)--(-1,0) node[above, midway]{Kauer move} node[below,midway]{$(i)$};
        \begin{scope}[xshift=2cm]
        \node[vertex] (1) at (-2,0) {};
        \node[vertex] (2) at (2,0) {};
        \node[vertex] (3) at (0,1.7) {};
        \node[vertex] (4) at (0,-1.7) {};
        \draw[very thick,Aquamarine] (3)--(4) node[right, midway]{$s'$};
        \draw (1)--(3);
        \draw (2)--(4);
        \draw[->] (-0.65,1) arc(230:265:0.95);
        \draw[->] (0.65,-1) arc(50:85:0.95);
        \end{scope}
        
        \begin{scope}[xshift=-7cm, yshift=-5cm]
        \node[vertex] (1) at (-2,0) {};
        \node[vertex] (2) at (2,0) {};
        \node[vertex] (3) at (0,1.7) {};
        \draw[very thick,Aquamarine] (1)--(2) node[above, midway]{$s$};
        \draw (1)--(3);
        \draw[->] (-1,0.1) arc (0:30:1);
        \end{scope}
        \draw[->] (-4,-5)--(-1,-5) node[above, midway]{Kauer move} node[below,midway]{$(ii)$};
        \begin{scope}[xshift=2cm, yshift=-5cm]
        \node[vertex] (1) at (-2,0) {};
        \node[vertex] (2) at (2,0) {};
        \node[vertex] (3) at (0,1.7) {};
        \draw[very thick,Aquamarine] (3)--(2) node[above right, midway]{$s'$};
        \draw (1)--(3);
        \draw[->] (-0.65,1) arc(230:315:0.95);
        \end{scope}

        \begin{scope}[xshift=-7cm, yshift=-10cm]
        \node[vertex] (1) at (-2,0) {};
        \node[vertex] (2) at (2,0) {};
        \draw[very thick,Aquamarine] (2.160) arc(-105:-434:0.8) node[above, midway]{$s$};
        \draw (1)--(2);
        \draw[<-] (1,0.1) arc(180:150:1);
        \end{scope}
        \draw[->] (-4,-10)--(-1,-10) node[above, midway]{Kauer move} node[below,midway]{$(iii)$};
        \begin{scope}[xshift=2cm, yshift=-10cm]
        \node[vertex] (1) at (-2,0) {};
        \node[vertex] (2) at (2,0) {};
        \draw[very thick,Aquamarine] (1.160) arc(-105:-434:0.8) node[above, midway]{$s'$};
        \draw (1)--(2);
        \draw[->] (-1,0.1) arc(0:30:1);
        \end{scope}
        \end{tikzpicture}
    \caption{Kauer move at the edge $s$}
    \label{Kauer moves}
\end{figure}
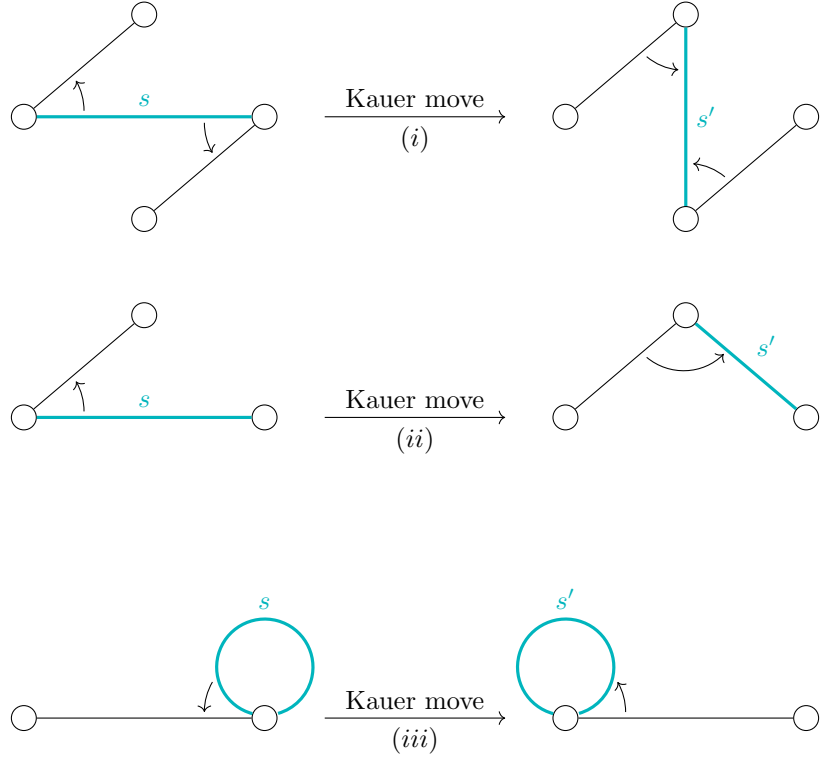

\smallskip

To generalize this notion of Kauer moves, we will use an alternative definition of a Brauer graph in terms of half-edges, which turns out to be equivalent to Definition \ref{def:Brauer graph}. This notion of Brauer graph is already used in combinatorial topology (see \cite{Lazarus} for instance) and is called an \textit{oriented map} in this context. In what follows, $\mathfrak{S}_{H}$ denotes the symmetric group of $H$.

\medskip

\begin{defn} \label{def:Brauer graph H}
A \textit{Brauer graph} is the data $\Gamma=(H,\iota,\sigma)$ where 

\smallskip

\begin{itemize}[label=\textbullet, font=\tiny]
\item $H$ is the set of half-edges,
\item $\iota\in\mathfrak{S}_{H}$ without fixed point satisfies $\iota^{2}=\id{H}$ : it is called the \textit{pairing},
\item $\sigma\in\mathfrak{S}_{H}$ is called the \textit{orientation}.
\end{itemize}

\end{defn}

\medskip

To the data $\Gamma=(H,\iota,\sigma)$, we can naturally associate a graph as follows : the vertex set is $\Gamma_{0}=H/\sigma$, the edge set is $\Gamma_{1}=H/\iota$ and the source map is the natural projection $s:H\rightarrow H/\sigma$. Moreover, the permutation $\sigma$ induces a cyclic ordering of the edges around each vertex. We keep previous convention for the orientation of a Brauer graph (see before Example \ref{ex:Brauer graph}).

\medskip

\begin{ex}
Let $\Gamma$ be the Brauer graph defined in Example \ref{ex:Brauer graph}. Using Definition \ref{def:Brauer graph H}, this Brauer graph may be seen as $\Gamma=(H,\iota,\sigma)$ 

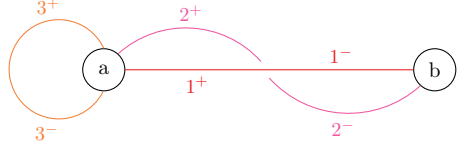
\smallskip

    \begin{figure}[H]
        \centering
        \begin{tikzpicture}[scale=1.1]
        \tikzstyle{vertex}=[minimum size=0.7cm,circle,draw,scale=0.8]
        \node[vertex] (a) at (-2,0) {a};
        \node[vertex] (b) at (2,0) {b};
        \draw[Red] (a)--(b) node[near end, above, scale=0.75]{$1^{-}$} node[near start, below, scale=0.75]{$1^{+}$};
        \draw[VioletRed] (a) to[bend left=45] node[midway, above, scale=0.75]{$2^{+}$} (-0.1,0.1); 
        \draw[VioletRed] (0,-0.1) to[bend right=45] node[midway, below, scale=0.75]{$2^{-}$} (b) ;
        \draw[Orange] (a.90) arc(25:333:0.6) node[near start, above, scale=0.75]{$3^{+}$} node[near end, below, scale=0.75]{$3^{-}$};
        \end{tikzpicture}
    \end{figure}

\smallskip

\noindent where $H=\{1^{+}, 1^{-}, 2^{+}, 2^{-}, 3^{+}, 3^{-}\}$ and $j^{-}$ denotes $\iota j^{+}$. Moreover, $\sigma=(1^{+} \ 2^{+} \ 3^{+} \ 3^{-})(1^{-} \ 2^{-})$. Each cycle in $\sigma$ represents a cyclic ordering of the half edges around a vertex.
\end{ex}

\medskip

\begin{rem}
    Interpreting a Brauer graph as a ribbon graph, one can associate to a Brauer graph a surface with boundary by means of a ribbon surface \cite[Subsection 2.7]{Schroll}. We can check that the set of faces of its surface corresponds to $H/\sigma \iota$.
\end{rem}

\medskip

For a fixed set of half-edges $H$ and a fixed pairing $\iota\in\mathfrak{S}_{H}$, we can define a left and a right action of $\mathfrak{S}_{H}$ on the set of Brauer graphs of the form $(H,\iota,\sigma)$ with $\sigma\in\mathfrak{S}_{H}$ given by

\begin{equation*}
    \begin{gathered}
    \begin{aligned}
    L:\mathfrak{S}_{H}&\rightarrow \{(H,\iota,\sigma')\,\vert \, \sigma'\in\mathfrak{S}_{H}\} \\
    \tau &\mapsto L_{\tau}(H,\iota,\sigma)=(H,\iota,\tau\sigma)
    \end{aligned} \\
    \begin{aligned}
        R:\mathfrak{S}_{H}&\rightarrow \{(H,\iota,\sigma')\,\vert \, \sigma'\in\mathfrak{S}_{H}\} \\
    \tau &\mapsto R_{\tau}(H,\iota,\sigma)=(H,\iota,\sigma\tau)
    \end{aligned}
    \end{gathered}
\end{equation*}

\noindent For $h_{1},h_{2}\in H$, the left and right action of the transposition $\tau=(h_{1} \ h_{2})$ on the Brauer graph $(H,\iota,\sigma)$ depends whether $h_{1}$ and $h_{2}$ are in the same $\sigma$-orbit or not. If they are in the same $\sigma$-orbit, we may see these actions as cutting actions. Else, we may see these actions as pasting actions. One can represent the left and right action of $\tau=(h_{1} \ h_{2})$ as follows

\smallskip

\begin{figure}[H]
\centering
\begin{tikzpicture}[scale=0.85]
        \tikzstyle{vertex}=[draw, circle, minimum size=0.2cm]
        \begin{scope}[xshift=-3cm]
        \node[vertex] (1) at (0,0) {};
        \draw[dashed,very thick,Aquamarine] (1) edge (1.75,0.8) edge (-1.75,-0.8);
        \draw (1)--(2,0);
        \draw (1)--(-2,0);
        \draw (1)--(1,1.5) node[above right]{$h_{1}$};
        \draw (1)--(-1,1.5);
        \draw (1)--(-1,-1.5) node[below left]{$h_{2}$};
        \draw (1)--(1,-1.5);
        \draw[->] (1,0.1) arc(5:50:1);
        \draw[->] (-1,-0.1) arc(185:230:1);
        \end{scope}
        \draw[<->] (0,0)--(3,0) node[above, midway]{$L_{(h_{1},h_{2})}(H,\iota,\sigma)$};
        \begin{scope}[xshift=6cm]
        \node[vertex] (1) at (-0.55,0) {};
        \node[vertex] (2) at (0.55,0) {};
        \draw[dashed,very thick,Aquamarine] (1)--(2);
        \draw (2)--(2,0) node[above, midway]{$\sigma^{-1}h_{1}$};
        \draw (1)--(-2,0) node[below, midway]{$\sigma^{-1}h_{2}$};
        \draw (1)--(0,1.5) node[above left]{$h_{1}$};
        \draw (1)--(-1.25,1.25);
        \draw (2)--(0,-1.5) node[below right]{$h_{2}$};
        \draw (2)--(1.25,-1.25) ;
        \end{scope}
        \begin{scope}[xshift=-3cm, yshift=-5cm]
        \node[vertex] (1) at (0,0) {};
        \draw[dashed,very thick,Aquamarine] (1) edge (0,1.75) edge (0,-1.75);
        \draw (1)--(2,0);
        \draw (1)--(-2,0);
        \draw (1)--(1,1.5) node[above right]{$h_{1}$};
        \draw (1)--(-1,1.5);
        \draw (1)--(-1,-1.5) node[below left]{$h_{2}$};
        \draw (1)--(1,-1.5);
        \draw[->] (0.5,0.88) arc(60:120:1);
        \draw[->] (-0.5,-0.88) arc(-120:-60:1);
        \end{scope}
        \draw[<->] (0,-5)--(3,-5) node[above, midway]{$R_{(h_{1},h_{2})}(H,\iota,\sigma)$};
        \begin{scope}[xshift=6cm, yshift=-5cm]
        \node[vertex] (1) at (-0.55,0) {};
        \node[vertex] (2) at (0.55,0) {};
        \draw[dashed,very thick,Aquamarine] (1)--(2);
        \draw (2)--(2,0);
        \draw (1)--(-2,0);
        \draw (2)--(1.2,1.5) node[above right]{$h_{1}$};
        \draw (1)--(-1.2,1.5) node[above]{$\sigma h_{1}$};
        \draw (1)--(-1.2,-1.5) node[below left]{$h_{2}$};
        \draw (2)--(1.2,-1.5) node[below]{$\sigma h_{2}$};
        \end{scope}
        \end{tikzpicture}
        \caption{Left and right action of $\tau=(h_{1} \ h_{2})$}
        \label{Left and right action}
\end{figure}
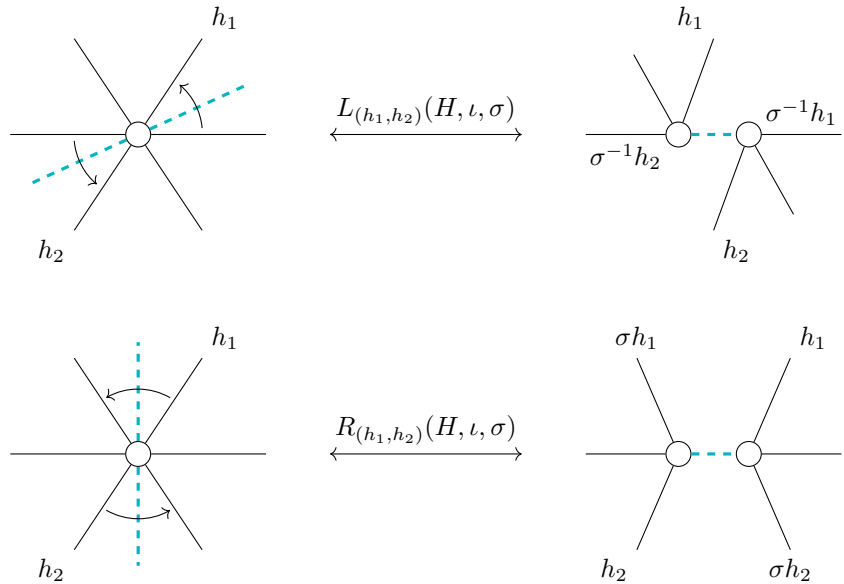

\smallskip

We now use these cutting and pasting actions to define a generalized Kauer move on any finite number of edges of a Brauer graph. For this, we will separate these edges into families of successive half-edges called sectors.

\medskip

\begin{defn} \label{def:sector}
    Let $\Gamma=(H,\iota,\sigma)$ be a Brauer graph and $H'$ be a subset of $H$ stable under $\iota$. 

    \smallskip

    \begin{itemize}[label=\textbullet, font=\tiny]
    \item We say that $(h,r)\in H\times \Z_{\ge 0}$ is a \textit{sector} in $\Gamma$ of elements in $H'$ if $r+1$ is the smallest integer $r'\ge0$ such that $\sigma^{r'}h\notin H'$.
    
    \item We say that $(h,r)\in H\times \Z_{\ge 0}$ is a \textit{maximal sector} in $\Gamma$ of elements in $H'$ if $(h,r)$ is a sector such that $\sigma^{-1}h\notin H'$.
    \end{itemize}

    \noindent We denote respectively $\mathrm{sect}(H',\sigma)$ and $\mathrm{Sect}(H',\sigma)$ the set of sectors and maximal sectors in $\Gamma$ of elements in $H'$.
\end{defn}

\medskip

Let $\Gamma=(H,\iota,\sigma)$ be a Brauer graph and $H'$ be a subset of $H$ stable under $\iota$. Thanks to the previous definition, we can define two permutations of $H$

\begin{equation*}
    \tau^{\Gamma}_{\mathrm{cut}}=\prod_{(h,r)\in\mathrm{Sect}(H',\sigma)}(h \ \ \ \sigma^{r+1}h) \qquad \mbox{and} \qquad \tau^{\Gamma}_{\mathrm{paste}}=\prod_{(h,r)\in\mathrm{Sect}(H',\sigma)}(\sigma^{r}h \ \ \ \iota\sigma^{r+1}h)
\end{equation*}

\noindent Note that these two products are well-defined since the transpositions have disjoint support. Using the cutting and pasting actions induced by the left and right actions $L$ and $R$ of $\mathfrak{S}_{H}$, we define a new Brauer graph 

\[\mu^{+}_{H'}(\Gamma)=R_{\tau^{\Gamma}_{\mathrm{paste}}}L_{\tau^{\Gamma}_{\mathrm{cut}}}(\Gamma)=(H,\iota,\tau^{\Gamma}_{\mathrm{cut}}\sigma\tau^{\Gamma}_{\mathrm{paste}})\]

\noindent that is locally obtained from $\Gamma$ as follows

\smallskip

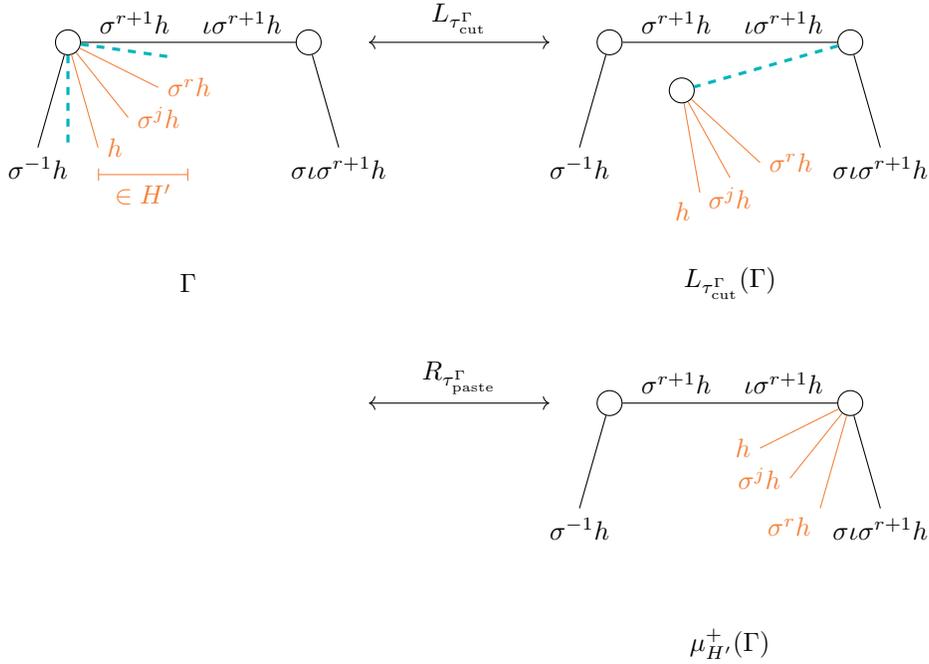
\begin{figure}[H]
\centering
        \begin{tikzpicture}[scale=0.8]
        \tikzstyle{vertex}=[draw, circle, minimum size=0.2cm]
        \begin{scope}[xshift=-3cm]
        \node[vertex] (1) at (-2,0) {};
        \node[vertex] (2) at (2,0) {};
        \draw (1)--(2) node[near start, above]{$\sigma^{r+1}h$} node[near end, above]{$\iota\sigma^{r+1}h$};
        \draw (1)--(-2.5,-1.75) node[below]{$\sigma^{-1}h$};
        \draw[Orange] (1)--(-1.5,-1.75) node[below, right]{$h$};
        \draw[Orange] (1)--(-1,-1.25) node[below, right]{$\sigma^{j}h$};
        \draw[Orange] (1)--(-0.5,-0.75) node[below, right]{$\sigma^{r}h$};
        \draw (2)--(2.5,-1.75) node[below]{$\sigma\iota\sigma^{r+1}h$};
        \draw[dashed,very thick,Aquamarine] (1) edge (-2,-1.75) edge (-0.25,-0.25);
        \draw (0,-4) node{$\Gamma$};
        \draw[|-|, Orange] (-1.5,-2.2)--(0,-2.2) node[below, midway]{$\in H'$};
        \end{scope} 
        
        \draw[<->] (0,0)--(3,0) node[above, midway]{$L_{\tau_{\mathrm{cut}}^{\Gamma}}$};
        
        \begin{scope}[xshift=6cm]
        \node[vertex] (1) at (-2,0) {};
        \node[vertex] (2) at (2,0) {};
        \node[vertex] (3) at (-0.8,-0.8) {};
        \draw (1)--(2) node[near start, above]{$\sigma^{r+1}h$} node[near end, above]{$\iota\sigma^{r+1}h$};
        \draw (1)--(-2.5,-1.75) node[below]{$\sigma^{-1}h$};
        \draw[Orange] (3)--(-0.5,-2.5) node[below left]{$h$};
        \draw[Orange] (3)--(0,-2.25) node[below]{$\sigma^{j}h$};
        \draw[Orange] (3)--(0.5,-2) node[below, right]{$\sigma^{r}h$};
        \draw (2)--(2.5,-1.75) node[below]{$\sigma\iota\sigma^{r+1}h$};
        \draw[dashed,very thick,Aquamarine] (2)--(3);
        \draw (0,-4) node{$L_{\tau^{\Gamma}_{\mathrm{cut}}}(\Gamma)$};
        \end{scope}

        \draw[<->] (0,-6)--(3,-6) node[above, midway]{$R_{\tau_{\mathrm{paste}}^{\Gamma}}$};

        \begin{scope}[xshift=6cm, yshift=-6cm]
        \node[vertex] (1) at (-2,0) {};
        \node[vertex] (2) at (2,0) {};
        \draw (1)--(2) node[near start, above]{$\sigma^{r+1}h$} node[near end, above]{$\iota\sigma^{r+1}h$};
        \draw (1)--(-2.5,-1.75) node[below]{$\sigma^{-1}h$};
        \draw[Orange] (2)--(1.5,-1.75) node[below left]{$\sigma^{r}h$};
        \draw[Orange] (2)--(1,-1.25) node[below, left]{$\sigma^{j}h$};
        \draw[Orange] (2)--(0.5,-0.75) node[below, left]{$h$};
        \draw (2)--(2.5,-1.75) node[below]{$\sigma\iota\sigma^{r+1}h$};
        \draw (0,-4) node{$\mu^{+}_{H'}(\Gamma)$};
        \end{scope}
        \end{tikzpicture}
\caption{Generalized Kauer move of $H'$}
\label{Generalized Kauer move}
\end{figure}

\medskip

\begin{defn} \label{def:generalized Kauer move}
    The \textit{generalized Kauer move} of $H'$ in $\Gamma$ is the local move defined previously, applied to the half-edges of $H'$ to obtain $\mu_{H'}^{+}(\Gamma)$ from $\Gamma$.
\end{defn}

\medskip

Using Figure \ref{Generalized Kauer move}, it is not hard to see that the following proposition holds.

\medskip

\begin{prop} \label{prop:usual Kauer move}
Let $\Gamma=(H,\iota,\sigma)$ be a Brauer graph. If $H'=\{h,\iota h\}$ for some $h\in H$, then the generalized Kauer move of $H'$ in $\Gamma$ is exactly the standard Kauer move at the edge associated to $h$ defined previously in Figure \ref{Kauer moves}.
\end{prop}

\medskip

Note that the generalized Kauer move of $H'$ is not necessarily the succession of the standard Kauer moves at each edge in $H'/\iota$ as pictured in the following example.

\medskip

\begin{ex} \label{ex:succession of standard Kauer moves 1} Let us consider the following Brauer graph $\Gamma=(H,\iota,\sigma)$

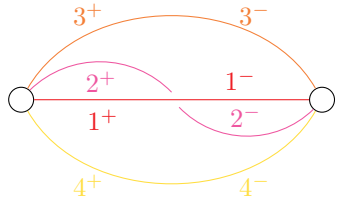
\smallskip

\begin{figure}[H]
    \centering
    \begin{tikzpicture}
        \tikzstyle{vertex}=[draw, circle, minimum size=0.2cm]
        \node[vertex] (1) at (-2,0) {};
        \node[vertex] (2) at (2,0) {};
        \draw[Red] (1)--(2) node[near start, below]{$1^{+}$} node[near end, above]{$1^{-}$};
        \draw[VioletRed] (1) to[bend left=45] node[midway, below]{$2^{+}$} (0,0.1);
        \draw[VioletRed] (0.1,-0.1) to[bend right=45] node[midway, above]{$2^{-}$} (2);
        \draw[Orange] (1) to[bend left=60] node[near start, above]{$3^{+}$} node[near end, above]{$3^{-}$} (2);
        \draw[Goldenrod] (1) to[bend right=60] node[near start, below]{$4^{+}$} node[near end, below]{$4^{-}$} (2);
        \end{tikzpicture}
    \label{Example 1}
\end{figure}

\smallskip

\noindent where $j^{-}$ denotes $\iota j^{+}$ and $\sigma=(1^{+} \ 2^{+} \ 3^{+} \ 4^{+})(1^{-} \ 2^{-} \ 4^{-} \ 3^{-})$. We set $H'=\{3^{+},3^{-}, 4^{+}, 4^{-}\}$. In this case, the maximal sectors of elements in $H'$ are $(3^{+},1)$ and $(4^{-},1)$. Thus, the Brauer graph obtained from $\Gamma$ by a generalized Kauer move of $H'$ is given by $\mu^{+}_{H'}(\Gamma)=(H,\iota,\sigma_{H'})$ where 

\[\sigma_{H'}=(3^{+} \ 1^{+})(4^{-} \ 1^{-})\sigma(4^{+} \ 1^{-})(3^{-} \ 1^{+})=(3^{+} \ 4^{+} \ 2^{-} \ 1^{-})(4^{-} \ 3^{-} \ 2^{+} \ 1^{+}) \ \]

\noindent Moreover, computing first the standard Kauer move at 3 and then at 4, we obtain the following Brauer graph $(H,\iota, \sigma_{4,3})$ where

\begin{equation*}
    \begin{aligned}
        \sigma_{4,3}
        &=(4^{+} \ 1^{+} )(4^{-} \ 3^{+})\left[(3^{+} \ 4^{+})(3^{-} \ 1^{-})\sigma(3^{+} \ 4^{-})(3^{-} \ 1^{+})\right](4^{+} \ 1^{-})(4^{-} \ 3^{-}) \\
        &=(4^{+} \ 2^{-} \ 3^{+} \ 1^{-})(4^{-} \ 2^{+} \ 1^{+} \ 3^{-})
    \end{aligned}
\end{equation*}

\noindent Similarly, computing first the standard Kauer move at 4 and then at 3, we obtain the following Brauer graph $(H,\iota, \sigma_{3,4})$ where

\[\sigma_{3,4}=(3^{+} \ 2^{-} \ 1^{-} \ 4^{+})(4^{-} \ 1^{+} \ 3^{-} \ 2^{+})\]

\noindent In particular, $\sigma_{H'}\neq \sigma_{3,4}$ and $\sigma_{H'}\neq\sigma_{4,3}$.

\end{ex}

\medskip

More generally, a generalized Kauer move of $H'$ is not necessarily a succession of standard Kauer moves as pictured in the following example.

\medskip

\begin{ex} \label{ex:succession of standard Kauer moves 2} Let us consider the following Brauer graph $\Gamma=(H,\iota, \sigma)$

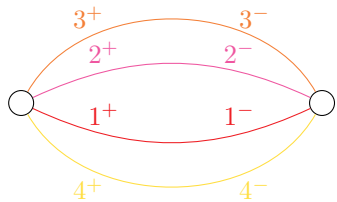
\smallskip

\begin{figure}[H]
    \centering
    \begin{tikzpicture}
        \tikzstyle{vertex}=[draw, circle, minimum size=0.2cm]
        \node[vertex] (1) at (-2,0) {};
        \node[vertex] (2) at (2,0) {};
        \draw[Red] (1) to[bend right=25] node[near start, above]{$1^{+}$} node[near end, above]{$1^{-}$} (2);
        \draw[VioletRed] (1) to[bend left=25] node[near start, above]{$2^{+}$} node[near end, above]{$2^{-}$} (2);
        \draw[Orange] (1) to[bend left=60] node[near start, above]{$3^{+}$} node[near end, above]{$3^{-}$} (2);
        \draw[Goldenrod] (1) to[bend right=60] node[near start, below]{$4^{+}$} node[near end, below]{$4^{-}$} (2);
        \end{tikzpicture}
    \label{Example 2}
\end{figure}

\smallskip

\noindent where $j^{-}$ denotes $\iota j^{+}$ and $\sigma=(1^{+} \ 2^{+} \ 3^{+} \ 4^{+})(1^{-} \ 4^{-} \ 3^{-} \ 2^{-})$. We set $H'=\{1^{+}, 1^{-}, 2^{+}, 2^{-}\}$. In this case, the maximal sectors of elements in $H'$ are $(1^{+},1)$ and $(2^{-},1)$. Thus, the Brauer graph obtained from $\Gamma$ by a generalized Kauer move of $H'$ is given by $\mu^{+}_{H'}(\Gamma)=(H,\iota,\sigma_{H'})$ where

\[\sigma_{H'}=(1^{+} \ 3^{+})(2^{-} \ 4^{-})\sigma(2^{+} \ 3^{-})(1^{-} 4^{+})=(1^{+} \ 2^{+} \ 4^{-} \ 3^{-})(2^{-} \ 1^{-} \ 3^{+} \ 4^{+})\]

\noindent Notice that any standard Kauer move in $\Gamma$ does not change the edge $j$ but interchanges the half-edges $j^{+}$ and $j^{-}$. In particular, since the arcs 2 and 4 are not direct successors in any vertex of $\Gamma$, we cannot obtain $\mu^{+}_{H'}(\Gamma)$ by a succession of standard Kauer moves.

\end{ex}

\medskip

Using the derived equivalence classification of Brauer graph algebras in \cite{OZ}, we can prove that the generalized Kauer moves also yield derived equivalences between Brauer graph algebras.

\medskip

\begin{prop} \label{prop:derived equivalent Brauer graph algebras}
    Let $\Gamma=(H,\iota,\sigma)$ be a Brauer graph and $H'$ be a subset of $H$ stable under $\iota$. Then, the Brauer graph algebras associated to $\Gamma$ and $\mu^{+}_{H'}(\Gamma)$ are derived equivalent.

\end{prop}

\medskip

\begin{proof} By Theorem 7.12 in \cite{OZ}, we have to check the following conditions

\smallskip

\begin{enumerate}[label=(\roman*)]
\item $\Gamma$ and $\mu^{+}_{H'}(\Gamma)$ have the same number of vertices and edges,
\item $\Gamma$ and $\mu^{+}_{H'}(\Gamma)$ have the same number of faces and their multi-sets of perimeters of the faces coincide,
\item Either both or none of $\Gamma$ and $\mu^{+}_{H'}(\Gamma)$ are bipartite.
\end{enumerate}

\smallskip

\noindent Recall that a graph $(V,E)$ is bipartite if its set of vertices $V$ admits a partition $V=V_{1}\coprod V_{2}$ such that each edge connects a vertex in $V_{1}$ with a vertex in $V_{2}$. Using Definition \ref{def:Brauer graph H} of a Brauer graph, these items may be reformulated as

\smallskip

\begin{enumerate}[label=(\roman*)]
\item \label{i} $\mathrm{card}(H/\sigma)=\mathrm{card}(H/\sigma_{H'})$ and $\mathrm{card}(H/\iota)=\mathrm{card}(H/\iota)$,
\item \label{ii} $\sigma\iota$ and $\sigma_{H'}\iota$ are conjugated in $\mathfrak{S}_{H}$,
\item \label{iii} $\exists C:H\rightarrow\{0,1\}$ such that $C\sigma=C \mbox{ and } C\iota=1-C \Leftrightarrow \exists C':H\rightarrow\{0,1\}$ such that $C'\sigma_{H'}=C' \mbox{ and } C'\iota=1-C'$.
\end{enumerate}

\smallskip

\noindent where $\sigma_{H'}=\tau^{\Gamma}_{\mathrm{cut}}\sigma\tau^{\Gamma}_{\mathrm{paste}}$ is the orientation of $\mu^{+}_{H'}(\Gamma)$. The item \ref{i} is clear : each transposition in $\tau_{\mathrm{cut}}^{\Gamma}$ corresponding to a maximal sector $(h,r)$,  creates a new vertex and the transposition in $\tau_{\mathrm{paste}}^{\Gamma}$ corresponding to $(h,r)$ withdraws this new vertex as pictured in Figure \ref{Generalized Kauer move}. Moreover, one can check that $t\tau^{\Gamma}_{\mathrm{cut}}$ is an element in $\mathfrak{S}_{H}$ conjugating $\sigma\iota$ and $\sigma_{H'}\iota$ in \ref{ii}, where $t\in\mathfrak{S}_{H}$ is defined by 

\[t_{\vert_{H'}}=(\tau^{\Gamma}_{\mathrm{cut}}\sigma\iota)^{-1}_{\vert_{H'}} \qquad \mbox{and} \qquad t_{\vert_{H\backslash H'}}=\id{H\backslash H'}\]

\noindent Finally, for item \ref{iii}, notice that $\Gamma$ can be constructed from $\mu^{+}_{H'}(\Gamma)$ using previous cutting and pasting actions. Thus, we only have to construct $C':H\rightarrow \{0,1\}$ from $C:H\rightarrow\{0,1\}$. Let us define $C'$ as follows

\[C'(h)=C(h) \ \mbox{if $h\notin H'$} \qquad \mbox{and} \qquad C'(h)=1-C(h) \ \mbox{if $h\in H$}\]

\noindent One can check that $C'\sigma_{H'}=C'$ and $C'\iota=1-C'$.

\end{proof}

\medskip

\begin{rem}
    In Example 4.7 in \cite{Antipov}, Antipov constructs the two following Brauer graphs $\Gamma_{1}=(H,\iota,\sigma_{1})$ and $\Gamma_{2}=(H,\iota,\sigma_{2})$

    \smallskip
    
    \begin{figure}[H]
        \centering
           \begin{tikzpicture}
        \tikzstyle{vertex}=[draw, circle, minimum size=0.2cm]
        \begin{scope}
        \node[vertex] (1) at (-2,0) {};
        \node[vertex] (2) at (2,0) {};
        \draw[Red] (1)--(2) node[near start, below]{$2^{+}$} node[very near end, above]{$2^{-}$};
        \draw[VioletRed] (1) to[bend right=60] node[midway, below]{$1^{+}$} (0,-0.1);
        \draw[VioletRed] (0.1,0.1) to[bend left=60] node[midway, above]{$1^{-}$} (2);
        \draw[Fuchsia] (1) to[bend left=45] node[midway, above right]{$3^{+}$} (-0.5,0.1);
        \draw[Fuchsia] (-0.35,-0.1) -- (-0.2,-0.2);
        \draw[Fuchsia] (0,-0.35) to[bend right=25] node[near start, above]{$3^{-}$} (2);
        \draw[Goldenrod] (1) to[bend right=90] node[near start, below]{$5^{+}$} node[near end, below]{$5^{-}$} (2);
        \draw[Orange] (1) to[bend left=60] node[midway, above]{$4^{+}$} (0.5,0.6);
        \draw[Orange] (0.65,0.4) to[bend left=5] (0.9,0.1);
        \draw[Orange] (1,-0.1) to[bend right=5] (1.2,-0.3);
        \draw[Orange] (1.4,-0.45) to[bend right] node[below left]{$4^{-}$}(2);
        \draw (0,-2) node{$\Gamma_{1}$};
        \end{scope}
        
        \begin{scope}[xshift=6cm]
        \node[vertex] (1) at (-2,0) {};
        \node[vertex] (2) at (2,0) {};
        \draw[Fuchsia] (1) to[bend left=60] node[near start, below]{$3^{+}$} node[near end, above right]{$3^{-}$} (2);
        \draw[VioletRed] (1) to[bend right=45] node[near end, below]{$1^{+}$} (0,0);
        \draw[VioletRed] (0,0) to[bend left=45] node[midway, above]{$1^{-}$} (2);
        \draw[Red] (1) to[bend left=30] node[midway, below]{$2^{+}$} (-0.1,0.1);
        \draw[Red] (0.1,-0.1) to[bend right=60] node[very near end,below right]{$2^{-}$} (2);
        \draw[Goldenrod] (1) to[bend right=45] node[near start,below]{$5^{+}$} node[very near end, below]{$5^{-}$}(1,-0.7);
        \draw[Goldenrod] (1.2,-0.45) to[bend left=5] (2);
        \draw[Orange] (0.6,0.25) to[bend right=5] node[midway, below]{$4^{-}$}(2);
        \draw[Orange] (0.3,0.4) to[bend left=15] (-0.4,1);
        \draw[Orange] (-0.5,1.2) to[bend right=70] node[midway, above]{$4^{+}$} (1);
        \draw (0,-2) node{$\Gamma_{2}$};        
        \end{scope}
        \end{tikzpicture}
        \label{Antipov example}
    \end{figure}

    \smallskip
    
    \noindent with $\sigma_{1}=(1^{+} \ 2^{+} \ 3^{+} \ 4^{+} \ 5^{+})(1^{-} \ 2^{-} \ 3^{-} \ 4^{-} \ 5^{-})$ and $\sigma_{2}=(1^{+} \ 2^{+} \ 3^{+} \ 4^{+} \ 5^{+})(1^{-} \ 4^{-} \ 5^{-} \ 2^{-} \ 3^{-})$, where $j^{-}$ denotes $\iota j^{+}$. Using the derived equivalence classification of Brauer graph algebras in \cite{OZ}, we can prove that the Brauer graph algebras associated to $\Gamma_{1}$ and $\Gamma_{2}$ are derived equivalent. However, $\Gamma_{2}$ cannot be obtained from $\Gamma_{1}$ by a succession of generalized Kauer moves. Indeed, note that in $\Gamma_{1}$, $\iota$ and $\sigma_{1}$ commute. One can check that this property is preserved under generalized Kauer moves since in this case, $(h,r)$ is a maximal sector if and only if $(\iota h, r)$ is a maximal sector. Since $\iota$ and $\sigma_{2}$ do not commute in $\Gamma_{2}$, we conclude that $\Gamma_{2}$ cannot be obtained from $\Gamma_{1}$ by a succession of generalized Kauer moves. Thus, the generalized Kauer moves are not sufficient to obtain all Brauer graph algebras derived equivalent to a given one.
    
\end{rem}

\medskip

\section{Link with gentle algebras}

\medskip

We have seen in Theorem \ref{thm:Schroll} that gentle algebras are in correspondence with Brauer graph algebras equipped with an admissible cut. Moreover, gentle algebras are also in correspondence with marked surfaces equipped with an admissible dissection \cite{OPS}. This last correspondence is based on the marked ribbon graph of a gentle algebra $\Lambda$, which is in fact the Brauer graph of its trivial extension $\mathrm{Triv}(\Lambda)$ equipped with the corresponding admissible cut \cite{Schroll}. The goal of this section is to define moves of admissible dissections in marked surfaces using generalized Kauer moves on Brauer graphs under previous correspondences. 

\medskip

\subsection{Admissible cuts and graded generalized Kauer moves}

\medskip

We want to understand what is an admissible cut using the notations of Definition \ref{def:Brauer graph H}. The choice of an admissible cut $C$ of a Brauer graph algebra $B=kQ/I$ gives a $\Z$-grading on its associated quiver $Q$ : the arrows in $C$ are of degree 1 and the others of degree 0. Since the relations \ref{item1:Brauer graph algebra}, \ref{item2:Brauer graph algebra} and \ref{item3:Brauer graph algebra} defining $I$ are homogeneous with respect to the $\Z$-grading on $Q$, this induces a $\Z$-grading on $B$. Moreover, the gentle algebra $\Lambda=B_{C}$ may be seen as the degree 0 subalgebra of $B$. Considering the source of the arrows, we can construct a bijection between the arrows of a Brauer graph algebra and the half-edges of its associated Brauer graph. Thus, an admissible cut $\Delta$ of a Brauer graph algebra $B$ is the data of a $\Z$-grading $d:H\rightarrow \{0,1\}$ on its associated Brauer graph $(H,\iota,\sigma)$ such that for all $v\in H/\sigma$ there exists a unique $h_{v}\in H$ whose source vertex is $v$ satisfying $d(h_{v})=1$. In fact, the admissible cuts belong to a particular class of $\Z$-gradings on a Brauer graph, called the $1$-homogeneous $\Z$-gradings. 

\medskip

\begin{defn}\label{def:1-homogeneous degree}
    Let $\Gamma=(H,\iota,\sigma)$ be a Brauer graph. 

    \smallskip
    
    \begin{itemize}[label=\textbullet, font=\tiny]
        \item A $\Z$-grading $d:H\rightarrow \Z$ is \textit{1-homogeneous} if for all $v\in H/\sigma$
    
    \[\sum_{h\in H, s(h)=v}d(h)=1\]

    \noindent where $s:H\rightarrow H/\sigma$ is the source map of $\Gamma$.
        \item We say that $(\Gamma,d)$ is a \textit{$\Z$-graded Brauer graph} if $d:H\rightarrow \Z$ is a 1-homogeneous $\Z$-grading on $\Gamma$.
    \end{itemize}
\end{defn}

\medskip

In particular, an admissible cut $d:H\rightarrow\Z$ of a Brauer graph $\Gamma=(H,\iota,\sigma)$ is a 1-homogeneous $\Z$-grading on $\Gamma$ whose values are contained in $\{0,1\}$. Let us now define a $\Z$-graded version of the generalized Kauer moves for $\Z$-graded Brauer graphs.

\medskip

\begin{defn}\label{def:graded generalized Kauer move}
    Let $(\Gamma,d)$ be a $\Z$-graded Brauer graph with $\Gamma=(H,\iota,\sigma)$. Let $H'$ be a subset of $H$ stable under $\iota$. The \textit{$\Z$-graded generalized Kauer move} of a sector $(h,r)\in\mathrm{sect}(H',\sigma)$ in $(\Gamma,d)$ gives the $\Z$-graded Brauer graph $\mu_{(h,r)}^{+}(\Gamma,d)=(H,\iota,\sigma_{(h,r)},d_{(h,r)})$ where

\[\sigma_{(h,r)}=(h \ \sigma^{r+1}h)\sigma(\sigma^{r}h \ \iota\sigma^{r+1}h)\]

\noindent where $d_{(h,r)}:H\rightarrow \Z$ is defined by

\begin{equation*}
    \begin{aligned}
        d_{(h,r)}: \ &\iota\sigma^{r+1}h &&\mapsto &&-\sum_{i=0}^{r}d(\sigma^{i}h) \\
        &\sigma^{r}h &&\mapsto &&\left\{\begin{aligned}
            &d(\iota\sigma^{r+1}h)+d(\sigma^{r}h) &&\mbox{if $\iota\sigma^{r+1}h\neq\sigma^{-1}h$} \\
            &\sum_{i=-1}^{r}d(\sigma^{i}h)+d(\sigma^{r}h) &&\mbox{else}
        \end{aligned}\right. \\
        &\sigma^{-1}h &&\mapsto &&\left\{\begin{aligned}
            &\sum_{i=-1}^{r}d(\sigma^{i}h) &&\mbox{if $\iota\sigma^{r+1}h\neq\sigma^{-1}h$} \\
            &-\sum_{i=0}^{r}d(\sigma^{i}h) &&\mbox{else}   
        \end{aligned}\right. \\
        &h' &&\mapsto &&d(h') \qquad \mbox{for $h'\neq\iota\sigma^{r+1}h,\sigma^{r}h,\sigma^{-1}h$}
    \end{aligned}
\end{equation*}

\smallskip

\begin{figure}[H]
    \centering
              
    \begin{tikzpicture}[scale=0.9]
        \tikzstyle{vertex}=[draw, circle, minimum size=0.2cm]
        \begin{scope}[xshift=-3cm]
        \node[vertex] (1) at (-2,0) {};
        \node[vertex] (2) at (2,0) {};
        \draw (1)--(2) node[near start, above]{$\sigma^{r+1}h$} node[near end, above]{$\iota\sigma^{r+1}h$};
        \draw (1)--(-2.5,-1.75) node[below]{$\sigma^{-1}h$};
        \draw[Orange] (1)--(-1.5,-1.75) node[below, right]{$h$};
        \draw[Orange] (1)--(-1,-1.25) node[below, right]{$\sigma^{j}h$};
        \draw[Orange] (1)--(-0.5,-0.75) node[below, right]{$\sigma^{r}h$};
        \draw (2)--(2.5,-1.75) node[below]{$\sigma\iota\sigma^{r+1}h$};
        \draw[|-|, Orange] (-1.5,-2.2)--(0,-2.2) node[below, midway]{$(h,r)$};
        \end{scope} 
        
        \draw[->] (0,0)--(3,0);

        \begin{scope}[xshift=6cm]
        \node[vertex] (1) at (-2,0) {};
        \node[vertex] (2) at (2,0) {};
        \draw (1)--(2) node[near start, above]{$\sigma^{r+1}h$} node[near end, above]{$\iota\sigma^{r+1}h$};
        \draw (1)--(-2.5,-1.75) node[below]{$\sigma^{-1}h$};
        \draw[Orange] (2)--(1.5,-1.75) node[below left]{$\sigma^{r}h$};
        \draw[Orange] (2)--(1,-1.25) node[below, left]{$\sigma^{j}h$};
        \draw[Orange] (2)--(0.5,-0.75) node[below, left]{$h$};
        \draw (2)--(2.5,-1.75) node[below]{$\sigma\iota\sigma^{r+1}h$};
        \end{scope}
        \end{tikzpicture}
    \caption{Generalized Kauer move of a sector (h,r)}
    \label{Generalized Kauer move of a sector}
\end{figure}
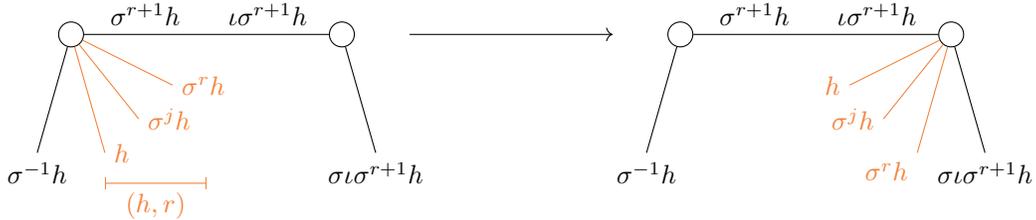
\end{defn}

\medskip

Let us check that $d_{(h,r)}:H\rightarrow \Z$ is indeed 1-homogeneous. Using the description of the cutting and pasting actions in Figure \ref{Left and right action}, there exists a bijection $\phi:H/\sigma\rightarrow H/\sigma_{(h,r)}$ satisfying 

\begin{equation*}
    \begin{aligned}
        &s_{(h,r)}(\sigma^{i}h)=\phi(s(\iota \sigma^{r+1}h))  &&\mbox{for $i=0,\ldots,r$} \\
        &s_{(h,r)}(h')=\phi(s(h'))  &&\mbox{for $h'\neq \sigma^{i}h$, $i=0,\ldots,r$}
    \end{aligned}
\end{equation*}

\noindent where $s:H\rightarrow H/\sigma$ and $s_{(h,r)}:H\rightarrow H/\sigma_{(h,r)}$ are the source map of $(H,\iota,\sigma)$ and $(H,\iota,\sigma_{(h,r)})$ respectively. Thus, it is easy to check that for all $v\in H/\sigma_{(h,r)}$

\[\sum_{h'\in H, s_{(h,r)}(h')=v}d_{(h,r)}(h')=\sum_{h'\in H, s(h')=\phi^{-1}(v)}d(h')=1\]

\medskip

\begin{rem} We have seen in Proposition \ref{prop:derived equivalent Brauer graph algebras} that the generalized Kauer moves over $H'$ yield derived equivalences between Brauer graph algebras. However, the generalized Kauer moves of a sector of elements in $H'$, defined in the previous definition, does not yield derived equivalences between Brauer graph algebras in general.
\end{rem}

\medskip

\begin{lem}\label{lem:successive sectors}
Let $(h_{1},r_{1}),(h_{2},r_{2})\in\mathrm{sect}(H',\sigma)$ be sectors in a Brauer graph $\Gamma=(H,\iota,\sigma)$ of elements in $H'\subset H$ stable under $\iota$, such that $\sigma^{i}h_{1}\neq \sigma^{j}h_{2}$ for $i=0,\ldots, r_{1}$, $j=0,\ldots, r_{2}$. Then, $(h_{1},r_{1})\in \mathrm{sect}(H',\sigma_{(h_{2},r_{2})})$.
\end{lem}

\medskip

\begin{proof}
By definition of $\sigma_{(h_{2},r_{2})}$, we have 

\[\sigma_{(h_{2},r_{2})}\sigma^{i}h_{1}=(h_{2} \ \sigma^{r_{2}+1}h_{2})\sigma(\sigma^{r_{2}}h_{2} \ \iota\sigma^{r_{2}+1}h_{2})\sigma^{i}h_{1}=\sigma^{i+1} h_{1}\]

\noindent for $i=0,\ldots, r_{1}$. In particular, $\sigma_{(h_{2},r_{2})}^{i}h_{1}=\sigma^{i}h_{1}$ for $i=0,\ldots,r_{1}+1$. 

\end{proof}

\medskip

For instance, this assumption holds if $(h_{1},r_{1})$ and $(h_{2},r_{2})$ are distinct maximal sectors in $\Gamma$. However, if $(h_{1},r_{1})$ and $(h_{2},r_{2})$ are distinct maximal sectors in $\Gamma$, $(h_{1},r_{1})$ is not necessarily a maximal sector in the Brauer graph $(H,\iota,\sigma_{(h_{2},r_{2})})$. Under the assumption of Lemma \ref{lem:successive sectors}, we can consider successive $\Z$-graded generalized Kauer moves. Let us denote 

\[\mu^{+}_{(h_{1},r_{1})(h_{2},r_{2})}(\Gamma,d)=(H,\iota,\sigma_{(h_{1},r_{1})(h_{2},r_{2})},d_{(h_{1},r_{1})(h_{2},r_{2})})\]

\noindent the $\Z$-graded Brauer graph defined by $\mu^{+}_{(h_{1},r_{1})}(\mu^{+}_{(h_{2},r_{2})}(\Gamma,d))$.

\medskip

\begin{prop}\label{prop:commutativity of graded generalized Kauer moves}
Let $(h_{1},r_{1}),(h_{2},r_{2})\in \mathrm{Sect}(H',\sigma)$ be distinct maximal sectors in a Brauer graph $\Gamma=(H,\iota,\sigma)$ of elements in $H'\subset H$ stable under $\iota$. Then, for any 1-homogeneous $\Z$-grading $d:H\rightarrow \Z$ of $\Gamma$

\[\mu^{+}_{(h_{1},r_{1})(h_{2},r_{2})}(\Gamma,d)=\mu^{+}_{(h_{2},r_{2})(h_{1},r_{1})}(\Gamma,d)\]
\end{prop}

\medskip

\begin{proof}
By definition, we have

\begin{equation*}
    \begin{aligned}
        \sigma_{(h_{1},r_{1})(h_{2},r_{2})}
        &=(h_{1} \ \sigma_{(h_{2},r_{2})}^{r_{1}+1}h_{1})\sigma_{(h_{2},r_{2})}(\sigma^{r_{1}}_{(h_{2},r_{2})}h_{1} \ \iota\sigma_{(h_{2},r_{2})}^{r_{1}+1}h_{1}) \\
        &=(h_{1} \ \sigma^{r_{1}+1}h_{1})(h_{2} \ \sigma^{r_{2}+1}h_{2})\sigma(\sigma^{r_{2}}h_{2} \ \iota\sigma^{r_{2}+1}h_{2})(\sigma^{r_{1}}h_{1} \ \iota\sigma^{r_{1}+1}h_{1}) \\
        &=(h_{2} \ \sigma^{r_{2}+1}h_{2})(h_{1} \ \sigma^{r_{1}+1}h_{1})\sigma(\sigma^{r_{1}}h_{1} \ \iota\sigma^{r_{1}+1}h_{1})(\sigma^{r_{2}}h_{2} \ \iota\sigma^{r_{2}+1}h_{2}) \\
        &=(h_{2} \ \sigma_{(h_{1},r_{1})}^{r_{2}+1}h_{2})\sigma_{(h_{1},r_{1})}(\sigma^{r_{2}}_{(h_{1},r_{1})}h_{2} \ \iota\sigma_{(h_{1},r_{1})}^{r_{2}+1}h_{2}) \\
        &=\sigma_{(h_{2},r_{2})(h_{1},r_{1})}
    \end{aligned}
\end{equation*}

\noindent It remains to check the commutativity of the degrees. By definition, the only half-edges whose degree is modified by the successive $\Z$-graded generalized Kauer moves are

\smallskip

\begin{itemize}[label=\textbullet, font=\tiny]
\item $\iota\sigma^{r_{i}+1}h_{i}$, $\sigma^{r_{i}}h_{i}$ and $\sigma^{-1}h_{i}$, for $i=1,2$;
\item $\iota\sigma_{(h_{j},r_{j})}^{r_{i}+1}h_{i}=\iota\sigma^{r_{i}+1}h_{i}$, for $\{i,j\}=\{1,2\}$;
\item $\sigma_{(h_{j},r_{j})}^{r_{i}}h_{i}=\sigma^{r_{i}}h_{i}$, for $\{i,j\}=\{1,2\}$;
\item $\sigma_{(h_{j},r_{j})}^{-1}h_{i}=\left\{\begin{aligned}
    &\sigma^{-1}h_{i} &&\mbox{if $\iota\sigma^{r_{j}+1}h_{j}\neq\sigma^{-1}h_{i}$} \\
    &\sigma^{r_{j}}h_{j} &&\mbox{else}
\end{aligned}\right.$, for $\{i,j\}=\{1,2\}$.
\end{itemize}

\smallskip

\noindent Since $h_{1}$ and $h_{2}$ play symmetric roles, it suffices to check the commutativity of the degrees for half-edges in the first line for $i=1$. We will only detail the computations for $\sigma^{r_{1}}h_{1}$, the other cases being similar.

\smallskip

\begin{itemize}[label=$\ast$]
    \item Assume that $\iota\sigma_{(h_{2},r_{2})}^{r_{1}+1}h_{1}\neq \sigma_{(h_{2},r_{2})}^{-1}h_{1}$. In this case $\iota\sigma^{r_{1}+1}h_{1}\neq\sigma^{-1}h_{1}$ and

\begin{equation*}
    \begin{aligned}
        d_{(h_{1},r_{1})(h_{2},r_{2})}(\sigma^{r_{1}}h_{1})
        &=d_{(h_{1},r_{1})(h_{2},r_{2})}(\sigma_{(h_{2},r_{2})}^{r_{1}}h_{1}) \\
        &=d_{(h_{2},r_{2})}(\iota\sigma_{(h_{2},r_{2})}^{r_{1}+1}h_{1})+d_{(h_{2},r_{2})}(\sigma^{r_{1}}_{(h_{2},r_{2})}h_{1}) \\
        &=d_{(h_{2},r_{2})}(\iota\sigma^{r_{1}+1}h_{1})+d_{(h_{2},r_{2})}(\sigma^{r_{1}}h_{1}) \\
        &=\left\{\begin{aligned}
            &d(\iota\sigma^{r_{1}+1}h_{1})+d(\sigma^{r_{1}}h_{1}) &&\mbox{if $\iota\sigma^{r_{1}+1}h_{1}\neq\sigma^{-1}h_{2}$} \\
            &\sum_{i=-1}^{r_{2}}d(\sigma^{i}h_{2})+d(\sigma^{r_{1}}h_{1}) &&\mbox{else}
        \end{aligned} \right.
    \end{aligned}
\end{equation*}

\noindent On the other hand, if $\iota\sigma^{r_{1}+1}h_{1}\neq\sigma^{-1}h_{2}$, 

\[d_{(h_{2},r_{2})(h_{1},r_{1})}(\sigma^{r_{1}}h_{1})=d_{(h_{1},r_{1})}(\sigma^{r_{1}}h_{1})=d(\iota\sigma^{r_{1}+1}h_{1})+d(\sigma^{r_{1}}h_{1})\]

\noindent Similarly, if $\iota\sigma^{r_{1}+1}h_{1}=\sigma^{-1}h_{2}$, then $\iota\sigma_{(h_{1},r_{1})}^{r_{1}+1}h_{2}\neq\sigma^{-1}_{(h_{1},r_{1})}h_{2}$ and

\begin{equation*}
    \begin{aligned}
        d_{(h_{2},r_{2})(h_{1},r_{1})}(\sigma^{r_{1}}h_{1})
        &=d_{(h_{2},r_{2})(h_{1},r_{1})}(\sigma_{(h_{1},r_{1})}^{-1}h_{2})    \\        &=\sum_{i=-1}^{r_{2}}d_{(h_{1},r_{1})}(\sigma^{i}_{(h_{1},r_{1})}h_{2}) \\
        &=\sum_{i=0}^{r_{2}}d_{(h_{1},r_{1})}(\sigma^{i}h_{2}) +d_{(h_{1},r_{1})}(\sigma^{r_{1}}h_{1}) \\
        &=\sum_{i=-1}^{r_{2}}d(\sigma^{i}h_{2})+d(\sigma^{r_{1}}h_{1}) 
    \end{aligned}
\end{equation*}

\item Assume that $\iota\sigma^{r_{1}+1}_{(h_{2},r_{2})}h_{1}=\sigma^{-1}_{(h_{2},r_{2})}h_{1}$. In this case $\iota\sigma^{r_{1}+1}h_{1}=\sigma^{-1}h_{1}$ and

\begin{equation*}
    \begin{aligned}
        d_{(h_{1},r_{1})(h_{2},r_{2})}(\sigma^{r_{1}}h_{1})
        &=d_{(h_{1},r_{1})(h_{2},r_{2})}(\sigma_{(h_{2},r_{2})}^{r_{1}}h_{1}) \\
        &=\sum_{i=-1}^{r_{1}}d_{(h_{2},r_{2})}(\sigma_{(h_{2},r_{2})}^{i}h_{1})+d_{(h_{2},r_{2})}(\sigma^{r_{1}}_{(h_{2},r_{2})}h_{1}) \\
        &=\sum_{i=-1}^{r_{1}}d(\sigma^{i}h_{1})+d(\sigma^{r_{1}}h_{1}) \\
        &=d_{(h_{1},r_{1})}(\sigma^{r_{1}}h_{1}) \\
        &=d_{(h_{2},r_{2})(h_{1},r_{1})}(\sigma^{r_{1}}h_{1})
    \end{aligned}
\end{equation*}
\end{itemize} 
\end{proof}

\medskip

Note that $\mu^{+}_{H'}(\Gamma)$ is the Brauer graph obtained by the succession of $\Z$-graded generalized Kauer moves of every maximal sectors in $\Gamma$. A priori, the 1-homogeneous $\Z$-grading $d_{(h,r)}:H\rightarrow\Z$ defined by a $\Z$-graded generalized Kauer move in a Brauer graph $\Gamma=(H,\iota,\sigma)$ equipped with an admissible cut is not necessarily an admissible cut of $(H,\iota,\sigma_{(h,r)})$. However, under some assumptions, we may ensure such a property.

\medskip

\begin{prop}\label{prop:admissible cut}
    Let $\Gamma=(H,\iota,\sigma)$ be a Brauer graph equipped with an admissible cut $d:H\rightarrow \Z$ and $H'\subset H$ stable under $\iota$. For any $(h,r)\in\mathrm{sect}(H',\sigma)$ such that $d(\sigma^{i}h)=0$ for $i=0,\ldots,r$, the 1-homogeneous $\Z$-grading $d_{(h,r)}:H\rightarrow \Z$ is an admissible cut of $(H,\iota,\sigma_{(h,r)})$.
\end{prop}

\medskip

\begin{proof}
    With our hypothesis, we have

    \begin{equation*}
        \begin{aligned}
            &d_{(h,r)}(\iota\sigma^{r+1}h)=0 \\
            &d_{(h,r)}(\sigma^{r}h)=d(\iota\sigma^{r+1}h) \\
            &d_{(h,r)}(\sigma^{-1}h)=\left\{\begin{aligned}
                &d(\sigma^{-1}h) &&\mbox{if $\iota\sigma^{r+1}h\neq\sigma^{-1}h$} \\
                &0 &&\mbox{else}
            \end{aligned}\right.
        \end{aligned}
    \end{equation*}

    \noindent In particular, the values of $d_{(h,r)}$ are contained in \{0,1\}. 
\end{proof}

\medskip

\subsection{Moves of admissible dissections}

\medskip

Let $\Gamma=(H,\iota,\sigma)$ be a Brauer graph equipped with an admissible cut $d:H\rightarrow \Z$. This data corresponds to a marked surface $(S,M)$ equipped with an admissible dissection $\Delta$. The surface $S$ is the ribbon surface of $\Gamma$ and the admissible cut $d$ of $\Gamma$ determines in which boundary component lie the marked points. Moreover, the admissible dissection $\Delta$ is the embedding of $\Gamma$ in $(S,M)$. We refer to \cite[Section 2]{OPS} for precise definitions. Under the assumptions of Proposition \ref{prop:admissible cut}, we want to describe the admissible dissection coming from a $\Z$-graded generalized Kauer move. Let us first recall some notions on marked surfaces from \cite{OPS} and \cite{APS}.

\medskip

\begin{defn}\label{def:marked surface}
A \textit{marked surface} is a pair $(S,M)$ where

\smallskip

\begin{itemize}[label=\textbullet, font=\tiny]
    \item $S$ is an oriented surface with non-empty boundary $\partial S$;
    \item $M=M_{\textcolor{green}{\circ}}\cup M_{\textcolor{red}{\bullet}}\cup P_{\textcolor{red}{\bullet}}$ is a finite set of marked points, where elements in $M_{\textcolor{green}{\circ}}\cup M_{\textcolor{red}{\bullet}}$ are in the boundary $\partial S$ of $S$ and elements in $P_{\textcolor{red}{\bullet}}$, called \textit{punctures}, are in the interior of $S$. Each boundary component of $\partial S$ is required to contain at least one marked point in $M_{\textcolor{green}{\circ}}\cup M_{\textcolor{red}{\bullet}}$ and the $\textcolor{green}{\circ}$ and $\textcolor{red}{\bullet}$-points are alternating on a given boundary component.
    \end{itemize}
\end{defn}

\medskip

On the surface, all curves are considered up to homotopy : we say that two curves intersect if any choice of homotopic representatives intersect.

\medskip

\begin{defn}\label{def:admissible dissection}

Let $(S,M)$ be a marked surface.

\smallskip

    \begin{itemize}[label=\textbullet, font=\tiny]
        \item A \textit{$\textcolor{green}{\circ}$-arc} is a non-contractible curve with endpoints in $M_{\textcolor{green}{\circ}}$.
        \item A collection of $\textcolor{green}{\circ}$-arcs is said to be  \textit{admissible} if the only possible intersections of two of these arcs are at the endpoints and there is at least one $\textcolor{red}{\bullet}$-point in $M_{\textcolor{red}{\bullet}}\cup P_{\textcolor{red}{\bullet}}$ in each subsurface enclosed by these arcs.
        \item An \textit{admissible dissection} is a maximal admissible collection of $\textcolor{green}{\circ}$-arcs.
    \end{itemize}
    
\end{defn}

\medskip

\begin{ex} \label{ex:marked surface}
Let $\Gamma=(H,\iota,\sigma)$ be the Brauer graph defined as follows

\smallskip

    
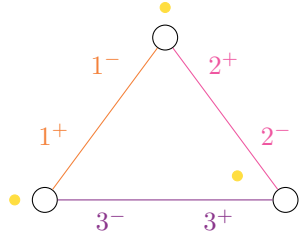
\begin{figure}[H]
        \centering
        
        \begin{tikzpicture}[scale=0.8]
            \tikzstyle{vertex}=[draw, circle, minimum size=0.2cm]
            \node[vertex] (1) at (-2,0) {};
            \node[vertex] (2) at (2,0) {};
            \node[vertex] (3) at (0,2.7) {};
            \draw[Orange] (1) to node[near start, above left]{$1^{+}$} node[near end, above left]{$1^{-}$} (3);
            \draw[VioletRed] (3) to node[near start, above right]{$2^{+}$} node[near end, above right]{$2^{-}$} (2);
            \draw[Fuchsia] (2) to node[near start, below]{$3^{+}$} node[near end, below]{$3^{-}$} (1) ;
            \node[draw, circle, Goldenrod, fill=Goldenrod, scale=0.4] at (0,3.2) {}; 
            \node[draw, circle, Goldenrod, fill=Goldenrod, scale=0.4] at (-2.5,0) {};
            \node[draw, circle, Goldenrod, fill=Goldenrod, scale=0.4] at (1.2,0.4) {};
        \end{tikzpicture}
    \end{figure}

    \smallskip

    \noindent where $j^{-}$ denotes $\iota j^{+}$ and $\sigma=(1^{-} \ 2^{+})(2^{-} \ 3^{+})(3^{-} \ 1^{+})$. We equip $\Gamma$ with an admissible cut $C$ which is represented in the previous figure by the yellow dots near each vertex. Let us denote $\Lambda$ the gentle algebra associated to $(\Gamma,C)$. In this case, the ribbon surface $S$ of $\Gamma$ is an annulus represented in the following figure

    \smallskip
    
    
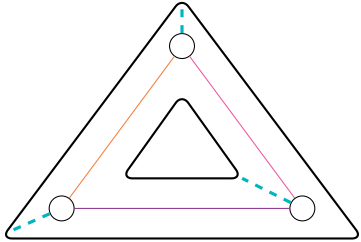
\begin{figure}[H]
        \centering
        
        \begin{tikzpicture}[scale=0.8]
            \tikzstyle{vertex}=[draw, circle, minimum size=0.2cm]
            \node[vertex] (1) at (-2,0) {};
            \node[vertex] (2) at (2,0) {};
            \node[vertex] (3) at (0,2.7) {};
            \draw[Orange] (1) -- (3);
            \draw[VioletRed] (3) -- (2);
            \draw[Fuchsia] (2) -- (1) ;
            \draw[thick, rounded corners] (-3,-0.5)--(0,3.5)--(3,-0.5)--cycle;
            \draw [thick, rounded corners] (-1, 0.5)--(0,1.9)--(1,0.5)--cycle;
            \draw[very thick, dashed, Aquamarine] (3.north)--(0,3.4);
            \draw[very thick, dashed, Aquamarine] (1)--(-2.9,-0.4);
            \draw[very thick, dashed, Aquamarine] (2)--(1,0.5);
        \end{tikzpicture}
    \end{figure}

    \smallskip

    \noindent In the previous figure, the blue dashed lines represent how the marked points in $S$ are attached to a boundary component thanks to the data of the admissible cut $C$. Thus, the marked surface with an admissible dissection corresponding to $\Lambda$ is 

    \smallskip
    
    
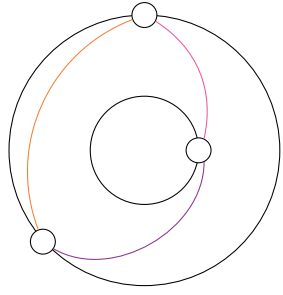
\begin{figure}[H]
        \centering
        \begin{tikzpicture}[scale=0.9]
            \tikzstyle{vertex}=[draw,circle,minimum size=0.2cm, fill=white]
            \draw (0,0) circle (2);
            \draw (0,0) circle (0.8);
            \node[vertex] (1) at (0,2) {};
            \node[vertex] (2) at (0.8,0) {};
            \node[vertex] (3) at (-1.5,-1.35) {};
            \draw[Orange] (1) to[bend right=45] (3);
            \draw[VioletRed] (1) to[bend left=35] (2.65);
            \draw[Fuchsia] (2.-65) to[bend left=60] (3);
        \end{tikzpicture}
    \end{figure}
\end{ex}

\medskip

Throughout the rest of the paper, we will adopt the convention in the last example to represent an admissible cut in a Brauer graph.

\medskip

\begin{nota}
For $h\in H$, we denote by $[h]\in H/\iota$ the edge associated to $h$ in $\Gamma$. Orienting the half-edges of $\Gamma$ towards their source vertex, we may consider $\overrightarrow{h}$, the oriented \textcolor{green}{$\circ$}-arc in $S$ associated to $[h]$ under the embedding of $\Gamma$ in $(S,M)$. Note that $\overrightarrow{\iota h}=(\overrightarrow{h})^{-1}$. Moreover, for $\gamma$ an oriented curve in $S$, $|\gamma|$ denotes the unoriented curve corresponding to $\gamma$. 
\end{nota}

\medskip

With the previous notations, the admissible dissection $\Delta$ of $(S,M)$ associated to the Brauer graph $\Gamma=(H,\iota,\sigma)$ equipped with an admissible cut $d:H\rightarrow\Z$ is given by

\[\Delta=\{|\overrightarrow{h}|\, , \, [h]\in H/\iota\}\]

Let $(h_{0},r_{0})\in\mathrm{sect}(H',\sigma)$ be a sector in $\Gamma$. We say that $h\in H$ \textit{belongs to} $(h_{0},r_{0})$ if there exists $j=0,\ldots,r_{0}$ such that $h=\sigma^{j}h_{0}$. Moreover, we define for all $h\in H$ the following $\textcolor{green}{\circ}$-arc 

\begin{equation*}
    \overline{h}:=\left\{ \begin{aligned}
        &|\overrightarrow{h}| &&\mbox{if $h$ and $\iota h$ do not belong to $(h_{0},r_{0})$} \\
        &|\overrightarrow{\iota\sigma^{r_{0}+1}h_{0}} \ \cdot \ \overrightarrow{h}| &&\mbox{if $h$ belongs to $(h_{0},r_{0})$ but $\iota h$ does not} \\
        &|\overrightarrow{\iota \sigma^{r_{0}+1}h_{0}} \ \cdot \ \overrightarrow{h} \ \cdot \  \overrightarrow{\sigma^{r_{0}+1}h_{0}}| &&\mbox{if $h$ and $\iota h$ belong to $(h_{0},r_{0})$}
    \end{aligned} \right.
\end{equation*}

\noindent where $\cdot$ is the concatenation defined before Proposition 1.20 in \cite{APS}.

\medskip

\begin{lem}\label{lem:move of admissible dissection}
    Let $(H,\iota,\sigma,d)$ be a Brauer graph equipped with an admissible cut $d:H\rightarrow \Z$ and $(S,M,\Delta)$ be the corresponding  marked surface equipped with an admissible dissection. Let $H'$ be a subset of $H$ stable under $\iota$ and $(h_{0},r_{0})\in\mathrm{sect}(H',\sigma)$ be a sector in $\Gamma=(H,\iota,\sigma)$ such that $d(\sigma^{i}h)=0$ for $i=0,\ldots,r_{0}$. Then,

    \smallskip
    
    \begin{enumerate}[label=(\arabic*)]
    \item \label{item:admissible dissection 1}The collection of arcs $\Delta'=\{\overline{h} \, , \, [h]\in H/\iota\}$ is an admissible dissection of $(S,M)$, 
    
    \item \label{item:admissible dissection 2}  The marked surface equipped with an admissible dissection corresponding to $\mu^{+}_{(h_{0},r_{0})}(\Gamma,d)$ is $(S,M,\Delta')$.
    \end{enumerate}
\end{lem}

\medskip

\begin{proof}
    The admissible dissection $\Delta$ of $(S,M)$ locally looks like

    \smallskip
    
    
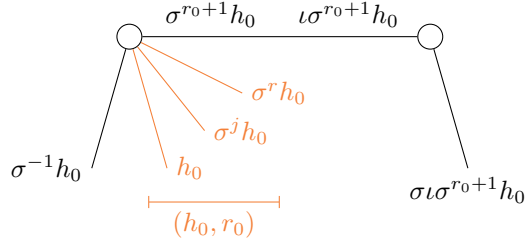
\begin{figure}[H]
        \centering
        \begin{tikzpicture}
        \tikzstyle{vertex}=[draw, circle, minimum size=0.2cm]
        \node[vertex] (1) at (-2,0) {};
        \node[vertex] (2) at (2,0) {};
        \draw (1)--(2) node[near start, above]{$\sigma^{r_{0}+1}h_{0}$} node[near end, above]{$\iota\sigma^{r_{0}+1}h_{0}$};
        \draw (1)--(-2.5,-1.75) node[below, left]{$\sigma^{-1}h_{0}$};
        \draw[Orange] (1)--(-1.5,-1.75) node[below, right]{$h_{0}$};
        \draw[Orange] (1)--(-1,-1.25) node[below, right]{$\sigma^{j}h_{0}$};
        \draw[Orange] (1)--(-0.5,-0.75) node[below, right]{$\sigma^{r}h_{0}$};
        \draw (2)--(2.5,-1.75) node[below]{$\sigma\iota\sigma^{r_{0}+1}h_{0}$};
        \draw[|-|, Orange] (-1.75,-2.2)--(-0,-2.2) node[below, midway]{$(h_{0},r_{0})$};        
        \end{tikzpicture}
        \label{Admissible dissection}
    \end{figure}

    \smallskip

    \noindent Considering a neighborhood of $|\overrightarrow{\sigma^{r_{0}+1}h_{0}}|$ in the interior of $S$ not containing any $\textcolor{Red}{\bullet}$-point of $M_{\textcolor{red}{\bullet}}\cup P_{\textcolor{red}{\bullet}}$, let us define another collection of $\textcolor{green}{\circ}$-arcs in $S$ that is locally given by

    \smallskip
    
    
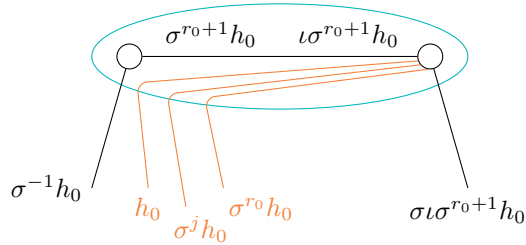
\begin{figure}[H]
        \centering
        \begin{tikzpicture}
        \tikzstyle{vertex}=[draw, circle, minimum size=0.2cm]
        \node[vertex] (1) at (-2,0) {};
        \node[vertex] (2) at (2,0) {};
        \draw (1)--(2) node[near start, above]{$\sigma^{r_{0}+1}h_{0}$} node[near end, above]{$\iota\sigma^{r_{0}+1}h_{0}$};
        \draw[Aquamarine] (0,0) ellipse (2.5cm and 0.7cm);
        \draw (1)--(-2.5,-1.75) node[below, left]{$\sigma^{-1}h_{0}$};
        \draw[Orange, rounded corners] (-1.75,-1.75) node[below]{$h_{0}$} --(-1.9,-0.35) --(2.-160);
        \draw[Orange, rounded corners] (-1.25,-2) node[below, xshift=2mm]{$\sigma^{j}h_{0}$} --(-1.5,-0.5) --(2.-140);
        \draw (2)--(2.5,-1.75) node[below]{$\sigma\iota\sigma^{r_{0}+1}h_{0}$};
        \draw[Orange, rounded corners] (-0.75,-1.75) node[below, xshift=5mm]{$\sigma^{r_{0}}h_{0}$} --(-1, -0.55) --(2.-100);
        \end{tikzpicture}
        \label{Kauer move of an admissible dissection}
    \end{figure}

    \smallskip

    \noindent Note that this new collection of $\textcolor{green}{\circ}$-arcs is the collection $\Delta'$ defined in \ref{item:admissible dissection 1}. By construction of $\Delta'$, it is easy to see that these $\textcolor{green}{\circ}$-arcs only intersect at endpoints and do not enclose a subsurface that does not contain any $\textcolor{red}{\bullet}$-point of $M_{\textcolor{red}{\bullet}}\cup P_{\textcolor{red}{\bullet}}$. This means that $\Delta'$ is an admissible collection of $\textcolor{green}{\circ}$-arcs. Since it has the same number of arcs as $\Delta$, it is in particular an admissible dissection. 
        
    For \ref{item:admissible dissection 2}, notice that there is a correspondence between the polygons enclosed by the $\textcolor{green}{\circ}$-arcs of $\Delta$ and $\Delta'$ given as follows

    
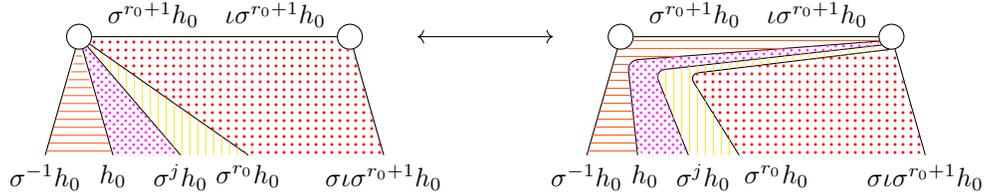
\begin{figure}[H]
        \centering
        \begin{tikzpicture}[scale=0.9]
        \tikzstyle{vertex}=[draw, circle, minimum size=0.2cm]
        \begin{scope}[xshift=-3cm]
        \node[vertex] (1) at (-2,0) {};
        \node[vertex] (2) at (2,0) {};
        \draw (1)--(2) node[near start, above]{$\sigma^{r_{0}+1}h_{0}$} node[near end, above]{$\iota\sigma^{r_{0}+1}h_{0}$};
        \draw (1)--(-2.5,-1.75) node[below]{$\sigma^{-1}h_{0}$};
        \draw (1)--(-1.5,-1.75) node[below]{$h_{0}$};
        \draw (1)--(-0.5,-1.75) node[below]{$\sigma^{j}h_{0}$};
        \draw (1)--(0.5,-1.75) node[below]{$\sigma^{r_{0}}h_{0}$};
        \draw (2)--(2.5,-1.75) node[below]{$\sigma\iota\sigma^{r_{0}+1}h_{0}$};     
        \fill[pattern=horizontal lines, pattern color=Orange] (-2.5,-1.75)--(1.-95) arc(-95:-85:0.2)--(-1.5,-1.75)--cycle;
        \fill[pattern=crosshatch dots, pattern color=VioletRed] (-1.5,-1.75)--(1.-70) arc(-70:-55:0.2)--(-0.5,-1.75)--cycle;
        \fill[pattern=vertical lines, pattern color=Goldenrod] (-0.5,-1.75)--(1.-45) arc(-45:-40:0.2)--(0.5,-1.75)--cycle;
        \fill[pattern=dots, pattern color=Red] (0.5,-1.75)--(1.-40) arc(-40:0:0.2) --(2.180) arc(-180:-80:0.2) --(2.5,-1.75)--cycle;
        \end{scope}
        
        \draw[<->] (0,0)--(2,0);
        
        \begin{scope}[xshift=5cm]
        \node[vertex] (1) at (-2,0) {};
        \node[vertex] (2) at (2,0) {};
        \draw (1)--(2) node[near start, above]{$\sigma^{r_{0}+1}h_{0}$} node[near end, above]{$\iota\sigma^{r_{0}+1}h_{0}$};
        \draw (1.-105)--(-2.5,-1.75) node[below]{$\sigma^{-1}h_{0}$};
        \draw[rounded corners] (-1.75,-1.75) node[below, xshift=1mm]{$h_{0}$} --(-1.9,-0.35) --(2.-160);
        \draw[rounded corners] (-1,-1.75) node[below, xshift=2mm]{$\sigma^{j}h_{0}$} --(-1.5,-0.5) --(2.-140);
        \draw (2)--(2.5,-1.75) node[below]{$\sigma\iota\sigma^{r_{0}+1}h_{0}$};
        \draw[rounded corners] (-0.25,-1.75) node[below, xshift=5mm]{$\sigma^{r_{0}}h_{0}$} --(-1, -0.55) --(2.-100); 
        \fill[pattern=horizontal lines, pattern color=Orange] (-2.5,-1.75)--(1.-105) arc(-105:0:0.2)--(2.-180) arc(-180:-160:0.2) to[rounded corners](-1.9,-0.35)--(-1.75,-1.75)--cycle;
        \fill[pattern=crosshatch dots, pattern color=VioletRed, rounded corners] (-1.75,-1.75)--(-1.9,-0.35)--(2.-160) arc(-160:-140:0.2)--(-1.5,-0.5)--(-1,-1.75)--cycle;
        \fill[pattern=vertical lines, pattern color=Goldenrod, rounded corners] (-1,-1.75)--(-1.5,-0.5)--(2.-140) arc(-140:-100:0.2)--(-1,-0.55)--(-0.25,-1.75)--cycle;
        \fill[pattern=dots, pattern color=Red, rounded corners] (-0.25,-1.75)--(-1,-0.55)--(2.-100) arc(-100:-80:0.2)--(2.5,-1.75)--cycle;
        \end{scope}
        \end{tikzpicture}
        \label{Correspondence of polygones}
    \end{figure}

    \noindent Since the $\textcolor{green}{\circ}$-arcs of $\Delta'$ correspond to the embedding of $(H,\iota,\sigma_{(h_{0},r_{0})})$ in $(S,M)$, it suffices to check that the admissible cut $d_{(h_{0},r_{0})}:H\rightarrow \Z$ is compatible with $d:H\rightarrow \Z$ under the previous polygon correspondence, meaning that for each vertex that is not incident to a single edge, the degree 1 arrow in a special cycle associated to $d$ and $d_{(h_{0},r_{0})}$ belongs to the same polygon under this correspondence. This is clear thanks to the expression of $d_{(h_{0},r_{0})}$ computed in the proof of Proposition \ref{prop:admissible cut}. 
\end{proof}

\medskip

Note that we can iterate the last result for sectors in $\Gamma$ satisfying the assumptions of Lemma \ref{lem:successive sectors} to compute the admissible dissection of $(S,M)$ corresponding to the succession of the $\Z$-graded generalized Kauer moves of these sectors.

\medskip

\begin{rem}
    Iterating Lemma \ref{lem:move of admissible dissection}, we can prove that 

    \[\mu^{+}_{(h_{0},r_{0})}(\Gamma,d)=\mu^{+}_{(h_{0},0)}\circ\ldots\circ\mu^{+}_{(\sigma^{r_{0}}h_{0},0)}(\Gamma,d)\]
\end{rem}

\medskip

\section{Compatibility with silting mutations}

\medskip

Standard Kauer moves may be understood in terms of silting mutations in the sense of \cite{AI}. The goal of this section is to interpret generalized Kauer moves in terms of silting mutations. 

\medskip

\subsection{Silting mutations}

\medskip

In this part, let us recall some definitions and results on silting mutations from \cite{AI}. Let us denote by $A$ a finite dimensional $k$-algebra. For any object $M$ in $\per{A}$, let $\mathrm{add}(M)$ be the full subcategory of $\per{A}$ consisting of direct sums of direct summands of $M$ and $\mathrm{thick}(M)$ be the smallest triangulated subcategory of $\per{A}$ containing $M$ that is closed under direct summands.

\medskip

\begin{defn}\label{def:silting and tilting category}
     Let $M$ be an object in $\per{A}$.

    \smallskip
    
     \begin{itemize}[label=\textbullet, font=\tiny]
        \item We say that $M$ is \textit{silting} if $\Hom[\per{A}]{M}{M[>0]}=0$ and $\mathrm{thick}(M)=\per{A}$.
        \item We say that $M$ is \textit{tilting} if $M$ is silting and $\Hom[\per{A}]{M}{M[<0]}=0$.
            
     \end{itemize}
\end{defn}

\medskip

\begin{defn}\label{def:left mutation}
   Let $M$ be an object in $\per{A}$.

   \smallskip

    \begin{itemize}[label=\textbullet, font=\tiny]
        \item We say that a morphism $f:X\rightarrow Y$ is a \textit{left $\mathrm{add}(M)$-approximation} of $X$ if $Y\in\mathrm{add}(M)$ and $\Hom[\per{A}]{f}{M}$ is surjective.
        \smallskip 
        
        \begin{center}
            \begin{tikzcd}[row sep=1cm, column sep=1cm]
                X \arrow[cramped]{r}[above]{f} \arrow[cramped]{d}[left]{\forall} &Y \arrow[cramped, dashed]{dl}{\exists} \\
                M &
            \end{tikzcd}
        \end{center}

        \smallskip
        \item Let $M_{0}\in\mathrm{add}(M)$. Let us consider the following triangle in $\per{A}$

        \smallskip
        \begin{center}
        \begin{tikzcd}
            M/M_{0} \arrow[cramped]{r}[above]{f} &M_{0}' \arrow[cramped]{r} &C_{M/M_{0}} \arrow[cramped]{r} &M/M_{0}[1] 
        \end{tikzcd}
        \end{center}

        \smallskip

        \noindent where $f:M/M_{0} \rightarrow M_{0}'$ is a left $\mathrm{add}(M_{0})$-approximation of $M/M_{0}$. Thus, the \textit{left mutation} of $M$ over $M_{0}$ is the following object of $\per{A}$

        \[\mu^{+}(M;M_{0}):=M_{0}\,\oplus\, C_{M/M_{0}}\]

        \end{itemize}
\end{defn}

\medskip

\begin{thm}[Aihara-Iyama {\cite[Theorem 2.32]{AI}}] \label{thm:AI}

\phantom{} 
\begin{itemize}[label=\textbullet, font=\tiny]

\item Any left mutation of a silting object of $\per{A}$ is a silting object of $\per{A}$.

\item Let $M$ be a tilting object of $\per{A}$ and $M_{0}\in\mathrm{add}(M)$. Then, $\mu^{+}(M;M_{0})$ is tilting if and only if $M/M_{0}$ has a left $\mathrm{add}(M_{0})$-approximation $f$ such that $\Hom[\per{A}]{M_{0}}{f}$ is injective.
\end{itemize}

\end{thm}

\medskip

\begin{rem}
    We can define dually a right approximation and a right mutation. Moreover, we have a dual version of Theorem \ref{thm:AI} for right mutations.
\end{rem}

\medskip

Theorem \ref{thm:AI} allows us to prove the following result on tilting mutations in symmetric algebras, which can be applied for Brauer graph algebras thanks to Theorem \ref{thm:Schroll}.

\medskip

\begin{prop} \label{prop:tilting mutation in symmetric algebra}
Let $A$ be a symmetric finite dimensional $k$-algebra and $e$ be an idempotent of $A$. Then, the left mutation $\mu^{+}(A;(1-e)A)$  is tilting in $\per{A}$.
\end{prop}

\medskip

\subsection{Compatibility for gentle algebras}

\medskip

In this part, we want to describe the silting mutation of a gentle algebra $\Lambda$ thanks to moves of the admissible dissection defining $\Lambda$. Let $\Lambda$ be a gentle algebra corresponding to a Brauer graph $\Gamma=(H,\iota,\sigma)$ equipped with an admissible cut $d:H\rightarrow\Z$. Let $B$ be the Brauer graph algebra associated to $\Gamma$. Recall that $\Lambda$ may be seen as the degree 0 subalgebra of $B$. For any subset $H'$ of $H$ stable under $\iota$, let us denote 

\[e_{H'}B=\bigoplus_{[h]\in H'/\iota}e_{[h]}B\]

\noindent where $e_{[h]}B$ is the indecomposable projective $B$-module associated to the edge $[h]\in H/\iota$ in $\Gamma$. Moreover, we denote by

\[\alpha(h,H'):e_{[h]}B\longrightarrow e_{[\sigma^{r(h)+1}h]}B\]

\noindent the morphism induced by the path in $B$ from $h$ to  $\sigma^{r(h)+1}h$, where $r(h)+1=\min\{r\ge 0\,\vert\,\sigma^{r}h\notin H'\}$. If $\sigma^{r}h\in H'$ for all $r\ge 0$, by an abuse of notation, we set $e_{[\sigma^{r(h)+1}h]}B=0$. Notice that for $h\notin H'$, $\alpha(h,H')=\id{e_{[h]}B}$. Let us denote $d(\alpha(h,H'))$ the degree of this morphism : it is the sum of $d(\sigma^{i}h)$ for $i=0,\ldots,r(h)$.

\medskip

\begin{prop}
    Let $H'$ be a subset of $H$ stable under $\iota$. For all $h\in H$, the left $\mathrm{add}(e_{H\backslash H'}B)$-approximation of $e_{[h]}B$ in $\per{B}$ is given by

    \smallskip
    
    \begin{center}
        \begin{tikzcd}[column sep=2cm]
            e_{[h]}B \arrow{r}{\begin{pmatrix} \alpha(h,H') \\ \alpha(\iota h,H')\end{pmatrix}} &e_{[\sigma^{r(h)+1}h]}B\oplus e_{[\sigma^{r(\iota h)+1}\iota h]}B
        \end{tikzcd}
    \end{center}
\end{prop}

\medskip

If $d(\alpha(h,H'))=0$, we denote by $\alpha(h,H')_{0}:e_{[h]}\Lambda\rightarrow e_{[\sigma^{r(h)+1}h]}\Lambda$ the morphism induced by $\alpha(h,H')$ on the degree 0 parts. 

\medskip

\begin{prop}[\cite{OPS}]\label{prop:OPS}
Let $(S,M,\Delta)$ be the marked surface equipped with an admissible dissection corresponding to $\Lambda$. Let $H'$ be a subset of $H$ stable under $\iota$ and let $h\in H$ such that $d(\alpha(h,H'))=d(\alpha(\iota h,H'))=0$. Then,

\smallskip

\begin{enumerate}[label=(\arabic*)]
    \item The left $\mathrm{add}(e_{H\backslash H'}\Lambda)$-approximation of $e_{[h]}\Lambda$ in $\per{\Lambda}$ is given by

    \smallskip
    
    \begin{center}
        \begin{tikzcd}[column sep=2cm]
            e_{[h]}\Lambda \arrow{r}{\begin{pmatrix} \alpha(h,H')_{0} \\ \alpha(\iota h,H')_{0}\end{pmatrix}} &e_{[\sigma^{r(h)+1}h]}\Lambda\oplus e_{[\sigma^{r(\iota h)+1}\iota h]}\Lambda
        \end{tikzcd}
    \end{center}

    \smallskip

    \item The cone of the previous morphism is an indecomposable string object in $\per{\Lambda}$ given in $(S,M)$ by the \textcolor{green}{$\circ$}-arc $|\overrightarrow{\iota\sigma^{r(h)+1}h} \ \cdot \ \overrightarrow{h} \ \cdot \ \overrightarrow{\sigma^{r(\iota h)+1}\iota h}|$.

    \smallskip

    \begin{figure}[H]
        \centering
         \begin{tikzpicture}[scale=0.9]
        \tikzstyle{vertex}=[draw, circle, minimum size=0.2cm]
        \node[vertex] (1) at (-2,0) {};
        \node[vertex] (2) at (2,0) {};
        \node[vertex] (3) at (0,1.7) {};
        \node[vertex] (4) at (-0,-1.7) {};
        \draw (1)--(2) node[very near start, below]{$h$} node[very near end, above]{$\iota h$};
        \draw (1)--(3) node[near start, above left]{$\sigma^{r(h)+1}h$} node[near end, above left]{$\iota\sigma^{r(h)+1}h$};
        \node[draw, circle, fill=Goldenrod, Goldenrod, scale=0.4] at (-2,-0.5) {};
        \draw (2)--(4) node[near start, below right]{$\sigma^{r(\iota h)+1}\iota h$} node[near end, below right]{$\iota\sigma^{r(\iota h)+1}\iota h$};
        \node[draw, circle, fill=Goldenrod, Goldenrod, scale=0.4] at (2,0.5) {};
        \draw[rounded corners, thick, dashed, Aquamarine] (3.-100)--(-1.5,0.15)--(1.5,-0.15) node[midway, above]{cone} --(4.80);
        \end{tikzpicture}
        \caption{Cone of the approximation of $e_{[h]}\Lambda$}
        \label{Cone of the approximation}
    \end{figure}
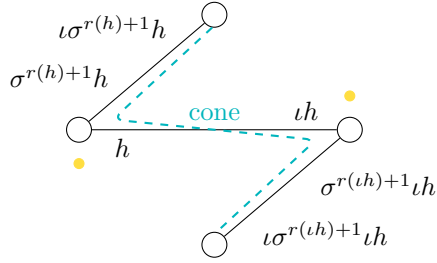

\end{enumerate}
\end{prop}

\medskip

We recall that the yellow dots in the previous figure represent the admissible cut $d$ defining $\Lambda$ as in Example \ref{ex:marked surface}. We can now use this description of the cones of approximations in 
$\Lambda$ to understand the left mutation of $\Lambda$ in terms of moves of its admissible dissection.

\medskip

\begin{thm}\label{thm:mutation in gentle algebra}
    Let $\Lambda$ be a gentle algebra corresponding to a Brauer graph $\Gamma=(H,\iota,\sigma)$ equipped with an admissible cut $d:H\rightarrow\Z$. Let $H'$ be a subset of $H$ stable under $\iota$ such that for all $h\in H'$, $d(\alpha(h,H'))=0$. Then,

    \smallskip
    
    \begin{enumerate}[label=(\arabic*)]
        \item \label{item1:mutation in gentle algebra} The left mutation $\mu^{+}(\Lambda;e_{H\backslash H'}\Lambda)$ is tilting in $\per{\Lambda}$.
        \item \label{item2:mutation in gentle algebra} $\mu^{+}_{\pi}(\Gamma,d)$ is a Brauer graph with an admissible cut, where $\pi$ denotes the product of all maximal sectors $(h,r)\in\mathrm{Sect}(H',\sigma)$ in $\Gamma$. The gentle algebra associated to $\mu^{+}_{\pi}(\Gamma,d)$ is the endomorphism algebra of $\mu^{+}(\Lambda;e_{H\backslash H'}\Lambda)$.
    \end{enumerate}
\end{thm}

\medskip

\begin{rem}    
    If $H'=\{h,\iota h\}$ for some $h\in H$ and $\Gamma$ has at least 2 edges, notice that we can prove this result thanks to Corollary 3.7 in \cite{CS}. However, we have seen in Example \ref{ex:succession of standard Kauer moves 2} that generalized Kauer moves cannot be described in general as a succession of standard Kauer moves. Thus, we cannot iterate the result from \cite{CS} to prove Theorem \ref{thm:mutation in gentle algebra}.
\end{rem}

\medskip

\begin{proof}
\ref{item1:mutation in gentle algebra} We will use Theorem \ref{thm:AI} to prove that $\mu^{+}(\Lambda;e_{H\backslash H'}\Lambda)$ is tilting. Let $h\in H'$. By Proposition \ref{prop:OPS}, recall that the left $\mathrm{add}(e_{H\backslash H'}\Lambda)$-approximation of $e_{[h]}\Lambda$ in $\per{\Lambda}$ is 

\smallskip

\begin{center}
        \begin{tikzcd}[column sep=2cm]
            e_{[h]}\Lambda \arrow{r}{\begin{pmatrix} \alpha(h,H')_{0} \\ \alpha(\iota h,H')_{0}\end{pmatrix}} &e_{[\sigma^{r(h)+1}h]}\Lambda\oplus e_{[\sigma^{r(\iota h)+1}\iota h]}\Lambda
        \end{tikzcd}
    \end{center}

\smallskip

\noindent Let $h_{0}\in H\backslash H'$ such that $\Hom[\per{\Lambda}]{e_{[h_{0}]}\Lambda}{e_{[h]}\Lambda}\neq0$ and $f:e_{[h_{0}]}\Lambda\rightarrow e_{[h]}\Lambda$ a morphism in $\per{\Lambda}$ such that the following diagram commutes 

\smallskip

\begin{center}
    \begin{tikzcd}[column sep=2cm, row sep=1.5cm]
            e_{[h]}\Lambda \arrow{r}{\begin{pmatrix} \alpha(h,H')_{0} \\ \alpha(\iota h,H')_{0}\end{pmatrix}} &e_{[\sigma^{r(h)+1}h]}\Lambda\oplus e_{[\sigma^{r(\iota h)+1}\iota h]}\Lambda \\
            e_{[h_{0}]}\Lambda \arrow{u}[left]{f} \arrow{ur}[below]{0}            
    \end{tikzcd}
\end{center}

\smallskip

\noindent We want to prove that $f=0$. Let $(S,M,\Delta)$ be the marked surface with an admissible dissection associated to $\Lambda$. Since $\Hom[\per{\Lambda}]{e_{[h_{0}]}\Lambda}{e_{[h]}\Lambda}\neq0$, the \textcolor{green}{$\circ$}-arcs $|\overrightarrow{h}|$ and $|\overrightarrow{h_{0}}|$ have an oriented intersection, which is necessarily at endpoints since $\Delta$ is an admissible dissection. We may assume for instance that $h_{0}$ is in the $\sigma$-orbit of $h$ in $\Gamma$. In particular, since $d(\alpha(h,H'))=0$, $e_{[\sigma^{r(h)+1}h]}\Lambda\neq 0$. Denoting respectively $\alpha$ and $\alpha'$ the paths in the Brauer graph algebra $B$ of $\Gamma$ associated to the morphisms $\alpha(h,H')_{0}$ and $\alpha(\iota h,H')_{0}$, the admissible dissection defining $\Lambda$ is locally given by one of the following.

\smallskip

\begin{itemize}[label=\textbullet, font=\tiny]
    \item \textit{1rst case} : $h$ and $\iota h$ are not in the same $\sigma$-orbit.
\end{itemize}

\smallskip

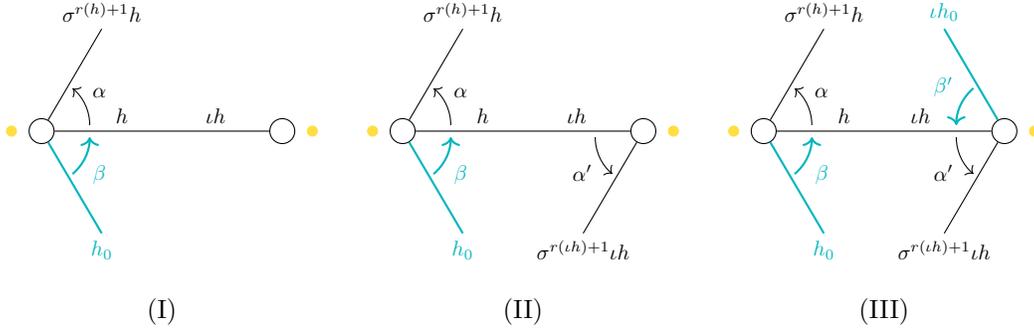
\begin{figure}[H]
    \centering
        \begin{tikzpicture}[scale=0.8]
        \tikzstyle{vertex}=[draw, circle, minimum size=0.1cm]
        \begin{scope}
        \node[vertex] (1) at (-2,0) {};
        \node[vertex] (2) at (2,0) {};
        \draw (1)--(2) node[near start, above right, scale=0.8]{$h$} node[near end, above, scale=0.8]{$\iota h$};
        \draw (1)--(-1,1.7) node[above, scale=0.8]{$\sigma^{r(h)+1}h$};
        \draw[thick, Aquamarine] (1)--(-1,-1.7) node[below, scale=0.8]{$h_{0}$};
        \draw[thick, Aquamarine, <-] (-1.2,-0.1) arc(0:-50:0.8) node[midway, below right, scale=0.8]{$\beta$} ;
        \draw[->] (-1.2,0.1) arc(0:50:0.8) node[midway, above right, scale=0.8]{$\alpha$};
        \draw (0,-3) node{(\Romannum{1})};
        \node[Goldenrod, draw, circle, fill=Goldenrod, scale=0.4] at (-2.5,0) {};
        \node[Goldenrod, draw, circle, fill=Goldenrod, scale=0.4] at (2.5,0) {};
        \end{scope}

        \begin{scope}[xshift=6cm]
        \node[vertex] (1) at (-2,0) {};
        \node[vertex] (2) at (2,0) {};
        \draw (1)--(2) node[near start, above right, scale=0.8]{$h$} node[near end, above, scale=0.8]{$\iota h$};
        \draw (1)--(-1,1.7) node[above, scale=0.8]{$\sigma^{r(h)+1}h$};
        \draw (2)--(1,-1.7) node[below, scale=0.8]{$\sigma^{r(\iota h)+1}\iota h$};
        \draw[thick, Aquamarine] (1)--(-1,-1.7) node[below, scale=0.8]{$h_{0}$};
        \draw[thick, Aquamarine, <-] (-1.2,-0.1) arc(0:-50:0.8) node[midway, below right, scale=0.8]{$\beta$} ;
        \draw[->] (-1.2,0.1) arc(0:50:0.8) node[midway, above right, scale=0.8]{$\alpha$};
        \draw[->] (1.2,-0.1) arc(180:230:0.8) node[midway,below left, scale=0.8]{$\alpha'$};
        \draw (0,-3) node{(\Romannum{2})};
        \node[Goldenrod, draw, circle, fill=Goldenrod, scale=0.4] at (-2.5,0) {};
        \node[Goldenrod, draw, circle, fill=Goldenrod, scale=0.4] at (2.5,0) {};
        \end{scope}

        \begin{scope}[xshift=12cm]
        \node[vertex] (1) at (-2,0) {};
        \node[vertex] (2) at (2,0) {};
        \draw (1)--(2) node[near start, above right, scale=0.8]{$h$} node[near end, above left, scale=0.8]{$\iota h$};
        \draw (1)--(-1,1.7) node[above, scale=0.8]{$\sigma^{r(h)+1}h$};
        \draw (2)--(1,-1.7) node[below, scale=0.8]{$\sigma^{r(\iota h)+1}\iota h$};
        \draw[thick, Aquamarine] (1)--(-1,-1.7) node[below, scale=0.8]{$h_{0}$};
        \draw[thick, Aquamarine, <-] (-1.2,-0.1) arc(0:-50:0.8) node[midway, below right, scale=0.8]{$\beta$} ;
        \draw[->] (-1.2,0.1) arc(0:50:0.8) node[midway, above right, scale=0.8]{$\alpha$};
        \draw[->] (1.2,-0.1) arc(180:230:0.8) node[midway,below left, scale=0.8]{$\alpha'$};
        \draw[thick, Aquamarine] (2)--(1,1.7) node[above, scale=0.8]{$\iota h_{0}$};
        \draw[thick, Aquamarine, <-]  (1.2,0.1) arc(180:130:0.8) node[midway, above left, scale=0.8]{$\beta'$};
        \draw (0,-3) node{(\Romannum{3})};
        \node[Goldenrod, draw, circle, fill=Goldenrod, scale=0.4] at (-2.5,0) {};
        \node[Goldenrod, draw, circle, fill=Goldenrod, scale=0.4] at (2.5,0) {};
        \end{scope}
        \end{tikzpicture}
        \end{figure}

\smallskip

\begin{itemize}[label=\textbullet, font=\tiny]
    \item \textit{2nd case} : $h$ and $\iota h$ are in the same $\sigma$-orbit. In this case, the edge associated to $h$ is a loop in $\Gamma$. Moreover, we denote $\gamma$ the path in $B$ that induces a morphism $f_{h}:e_{[h]}\Lambda\rightarrow e_{[h]}\Lambda$ in $\per{\Lambda}$ such that the endomorphism algebra of $e_{[h]}\Lambda$ is generated by $\id{e_{[h]}\Lambda}$ and $f_{h}$. 
    \end{itemize}

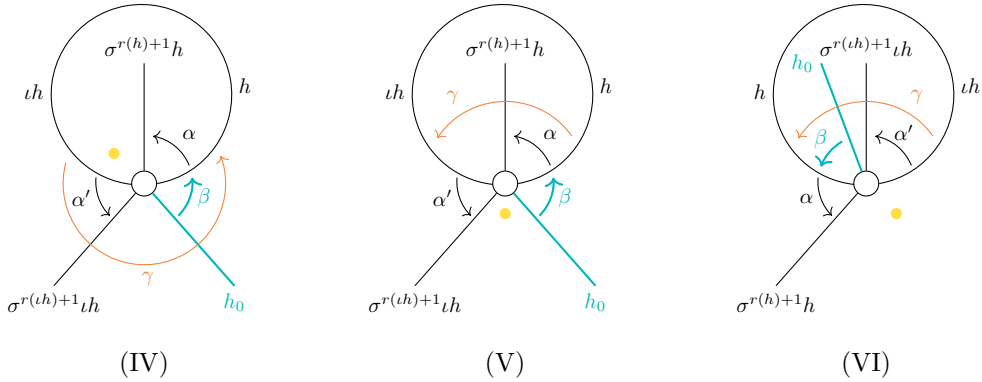
\smallskip

        \begin{figure}[H]
        \centering
        \begin{tikzpicture}[scale=0.8]
        \tikzstyle{vertex}=[draw, circle, minimum size=0.1cm]
        \begin{scope}[yshift=-8cm]
        \node[vertex] (1) at (0,0) {};
        \draw (1.0) arc (-80:264:1.5) node[near start, right, scale=0.8]{$h$} node[near end, left, scale=0.8]{$\iota h$};
        \draw (1.90)--(0,2) node[above, scale=0.8]{$\sigma^{r(h)+1}h$};
        \draw[thick, Aquamarine] (1)--(1.5,-1.7) node[below, scale=0.8]{$h_{0}$};
        \draw[thick, Aquamarine, <-] (0.8,0.1) arc(5:-45:0.8) node[midway, right, scale=0.8]{$\beta$};
        \draw (1)--(-1.5,-1.7) node[below, scale=0.8]{$\sigma^{r(\iota h)+1}\iota h$};
        \draw[->] (0.75,0.3) arc(22:80:0.8) node[midway, above right, scale=0.8]{$\alpha$};
        \node[Goldenrod, draw, circle, fill=Goldenrod, scale=0.4] at (-0.5,0.5) {};
        \draw[->] (-0.8,0.1) arc(175:225:0.8) node[midway,left, scale=0.8]{$\alpha'$};
        \draw[->, Orange] (-1.3,0.35) arc(165:380:1.35) node[midway, below, scale=0.8]{$\gamma$};
        \draw (0,-3) node{(\Romannum{4})};
        \end{scope}

        \begin{scope}[yshift=-8cm, xshift=6cm]
        \node[vertex] (1) at (0,0) {};
        \draw (1.0) arc (-80:264:1.5) node[near start, right, scale=0.8]{$h$} node[near end, left, scale=0.8]{$\iota h$};
        \draw (1.90)--(0,2) node[above, scale=0.8]{$\sigma^{r(h)+1}h$};
        \draw[thick, Aquamarine] (1)--(1.5,-1.7) node[below, scale=0.8]{$h_{0}$};
        \draw[thick, Aquamarine, <-] (0.8,0.1) arc(5:-45:0.8) node[midway, right, scale=0.8]{$\beta$};
        \draw (1)--(-1.5,-1.7) node[below, scale=0.8]{$\sigma^{r(\iota h)+1}\iota h$};
        \draw[->] (0.75,0.3) arc(22:80:0.8) node[midway, above right, scale=0.8]{$\alpha$};
        \node[Goldenrod, draw, circle, fill=Goldenrod, scale=0.4] at (0,-0.5) {};
        \draw[->] (-0.8,0.1) arc(175:225:0.8) node[midway,left, scale=0.8]{$\alpha'$};
        \draw[<-, Orange] (-1.13,0.75) arc(147:35:1.35) node[near start, above left, scale=0.8]{$\gamma$};
        \draw (0,-3) node{(\Romannum{5})};
        \end{scope}

        \begin{scope}[yshift=-8cm, xshift=12cm]
        \node[vertex] (1) at (0,0) {};
        \draw (1.0) arc (-80:264:1.5) node[near start, right, scale=0.8]{$\iota h$} node[near end, left, scale=0.8]{$h$};
        \draw (1.90)--(0,2) node[above, scale=0.8]{$\sigma^{r(\iota h)+1}\iota h$};
        \draw[thick, Aquamarine] (1)--(-0.75,2) node[left, scale=0.8]{$h_{0}$};
        \draw[thick, Aquamarine, ->] (-0.4,0.7) arc(120:160:0.8) node[near start,yshift=1mm,left, scale=0.8]{$\beta$};
        \draw (1)--(-1.5,-1.7) node[below, scale=0.8]{$\sigma^{r(h)+1}h$};
        \draw[->] (-0.8,0.1) arc(175:225:0.8) node[midway,left, scale=0.8]{$\alpha$};
        \node[Goldenrod, draw, circle, fill=Goldenrod, scale=0.4] at (0.5,-0.5) {};
        \draw[->] (0.75,0.3) arc(22:80:0.8) node[near end,yshift=1mm, right,scale=0.8]{$\alpha'$};
        \draw[Orange, <-] (-1.13,0.75) arc(147:35:1.35) node[near end, above right, scale=0.8]{$\gamma$};
        \draw (0,-3) node{(\Romannum{6})};
        \end{scope}
        \end{tikzpicture}
        \end{figure}

        \smallskip

        \begin{figure}[H]
        \centering
        \begin{tikzpicture}[scale=0.8]
        \tikzstyle{vertex}=[draw, circle, minimum size=0.2cm]
        \begin{scope}[xshift=3cm, yshift=-16cm]
        \node[vertex] (1) at (0,0) {};
        \draw (1.0) arc (-80:264:1.5) node[near start, right, scale=0.8]{$\iota h$} node[near end, left, scale=0.8]{$h$};
        \draw (1.90)--(0,2) node[above, scale=0.8]{$\sigma^{r(\iota h)+1}\iota h$};
        \draw[thick, Aquamarine] (1)--(-0.75,2) node[left, scale=0.8]{$h_{0}$};
        \draw[thick, Aquamarine, ->] (-0.4,0.7) arc(120:160:0.8) node[near start,yshift=1mm,left, scale=0.8]{$\beta$};
        \draw (1)--(-1.5,-1.7) node[below, scale=0.8]{$\sigma^{r(h)+1}h$};
        \draw[->] (0.75,0.3) arc(22:80:0.8) node[near end,yshift=1mm, right,scale=0.8]{$\alpha'$};
        \node[Goldenrod, draw, circle, fill=Goldenrod, scale=0.4] at (-0.27,1.5) {};
        \draw[->] (-0.8,0.1) arc(175:225:0.8) node[midway,left, scale=0.8]{$\alpha$};
        \draw[->, Orange] (-1.3,0.35) arc(165:380:1.35) node[midway, below, scale=0.8]{$\gamma$};
        \draw (0,-3) node{(\Romannum{7})};
        \end{scope}

         \begin{scope}[yshift=-16cm, xshift=9cm]
        \node[vertex] (1) at (0,0) {};
        \draw (1.0) arc (-80:264:1.5) node[near start, right, scale=0.8]{$h$} node[near end, left, scale=0.8]{$\iota h$};
        \draw[thick, Aquamarine] (1)--(1.5,-1.7) node[below, scale=0.8]{$h_{0}$};
        \draw[thick, Aquamarine, <-] (0.8,0.1) arc(5:-45:0.8) node[midway, right, scale=0.8]{$\beta$};
        \draw (1)--(-1.5,-1.7) node[below, scale=0.8]{$\sigma^{r(h)+1}\iota h$};
        \draw[->] (0.75,0.3) arc(22:220:0.8) node[midway, above, scale=0.8]{$\alpha$};
        \node[Goldenrod, draw, circle, fill=Goldenrod, scale=0.4] at (0,-0.5) {};
        \draw[->] (-1.3,0.35) arc(165:220:1.35) node[midway, left, scale=0.8]{$\alpha'$} ;
        \draw[Orange, <-] (-1.13,0.75) arc(147:35:1.35) node[midway, above, scale=0.8]{$\gamma$};
        \draw (0,-3) node{(\Romannum{8})};
        \end{scope}
        \end{tikzpicture}
\end{figure}

\smallskip

\noindent 
In the previous figures, $\beta$ denotes the path in $B$ associated to $f_{\beta}:e_{[h_{0}]}\Lambda\rightarrow e_{[h]}\Lambda$ a non zero morphism in $\per{\Lambda}$ coming from an oriented intersection of the $\textcolor{green}{\circ}$-arcs $|\overrightarrow{h}|$ and $|\overrightarrow{h_{0}}|$. Note that the $\textcolor{green}{\circ}$-arc $|\overrightarrow{h_{0}}|$ could be a loop, except for case (\Romannum{3}), and would induce in this case a morphism $f_{h_{0}}:e_{[h_{0}]}\Lambda\rightarrow e_{[h_{0}]}\Lambda$ in $\per{\Lambda}$ independent from $\id{e_{[h_{0}]}\Lambda}$. We will only prove that $f=0$ in case (\Romannum{5}), other cases being similar. In this case, the admissible dissection defining $\Lambda$ is locally given by one of the following

\smallskip

\begin{figure}[H]
    \centering
            \begin{tikzpicture}[scale=0.8]
        \tikzstyle{vertex}=[draw, circle, minimum size=0.1cm]
        \begin{scope}
        \node[vertex] (1) at (0,0) {};
        \draw (1.0) arc (-80:264:1.5) node[near start, right, scale=0.8]{$h$} node[near end, left, scale=0.8]{$\iota h$};
        \draw (1.90)--(0,2) node[above, scale=0.8]{$\sigma^{r(h)+1}h$};
        \draw[thick, Aquamarine] (1)--(1.5,-1.7) node[below, scale=0.8]{$h_{0}$};
        \draw (1)--(-1.5,-1.7) node[below, scale=0.8]{$\sigma^{r(\iota h)+1}\iota h$};
        \node[Goldenrod, draw, circle, fill=Goldenrod, scale=0.4] at (0,-0.5) {};
        \draw (0,-3) node{(\romannum{1})};
        \draw[thick, Aquamarine, <-] (0.8,0.1) arc(5:-45:0.8) node[midway, right, scale=0.8]{$\beta$};
        \draw[->] (0.75,0.3) arc(22:80:0.8) node[midway, above right, scale=0.8]{$\alpha$};
        \draw[->] (-0.8,0.1) arc(175:225:0.8) node[midway,left, scale=0.8]{$\alpha'$};
        \draw[<-, Orange] (-1.13,0.75) arc(147:35:1.35) node[near start, above left, scale=0.8]{$\gamma$};
        \end{scope}

        \begin{scope}[xshift=8cm]
        \node[vertex] (1) at (0,0) {};
        \draw (1.0) arc (-80:264:1.5) node[near start, right, scale=0.8]{$h$} node[near end, left, scale=0.8]{$\iota h$};
        \draw (1.90)--(0,2) node[above, scale=0.8]{$\sigma^{r(h)+1}h$};
        \draw[thick, Aquamarine, rounded corners] (1.-110) to[bend right=15] (0.5,-2) node[left, scale=0.8]{$\iota h_{0}$} to[bend right=45]  (1.5,-1.5) node[right, scale=0.8]{$h_{0}$} to[bend right=15] (1.-20);
        \draw[thick, Aquamarine, <-] (0.8,0.1) arc(5:-25:0.8) node[midway, right, scale=0.8]{$\beta$};
        \draw (1)--(-1.5,-1.7) node[below, scale=0.8]{$\sigma^{r(\iota h)+1}\iota h$};
        \draw[->] (0.75,0.3) arc(22:80:0.8) node[midway, above right, scale=0.8]{$\alpha$};
        \node[Goldenrod, draw, circle, fill=Goldenrod, scale=0.4] at (-0.3,-0.75) {};
        \draw[->] (-0.8,0.1) arc(175:225:0.8) node[midway,left, scale=0.8]{$\alpha'$};
        \draw[<-, Orange] (-1.13,0.75) arc(147:35:1.35) node[near start, above left, scale=0.8]{$\gamma$};
        \draw[->, thick, Aquamarine] (0.05,-0.8) arc(-85:-35:0.8) node[midway, below, scale=0.8]{$\gamma_{0}$};
        \draw (0,-3) node{(\romannum{2})};
        \end{scope}

        \end{tikzpicture}
\end{figure}

\smallskip

\noindent where $\gamma_{0}$ is the path in $B$ associated to $f_{h_{0}}:e_{[h_{0}]}\Lambda\rightarrow e_{[h_{0}]}\Lambda$. Using the correspondence between oriented intersections and bases of morphism spaces from \cite{OPS}, we deduce that

\begin{equation*}
    \Hom[\per{\Lambda}]{e_{[h_{0}]}\Lambda}{e_{[h]}\Lambda}=\left\{
    \begin{aligned}
        &\mathrm{Span}_{k}(f_{\beta},f_{h}f_{\beta}) &&\mbox{in case (i)} \\
        &\mathrm{Span}_{k}(f_{\beta},f_{h}f_{\beta}, f_{\beta}f_{h_{0}}, f_{h}f_{\beta}f_{h_{0}}) &&\mbox{in case (ii)}
    \end{aligned}\right.
\end{equation*}

\noindent Then, we can write

\begin{equation*}
    f=\left\{\begin{aligned}
    &\lambda f_{\beta}+\mu f_{h}f_{\beta} &&\mbox{in case (i)} \\
    &\lambda f_{\beta}+\mu f_{h}f_{\beta}+\lambda' f_{\beta} f_{h_{0}}+\mu' f_{h}f_{\beta}f_{h_{0}} &&\mbox{in case (ii)}
    \end{aligned}\right.
\end{equation*}

\noindent for some $\lambda,\lambda',\mu,\mu'\in k$. Moreover, since $\alpha\gamma=0$ and $\alpha'\beta=0$, we have 

\begin{equation*}
    \begin{aligned}
        &0=\alpha(h,H')_{0}f=\left\{\begin{aligned}&\lambda\alpha(h,H')_{0}f_{\beta} &&\mbox{in case (i)} \\
        &\lambda\alpha(h,H')_{0}f_{\beta}+\lambda'\alpha(h,H')_{0}f_{\beta}f_{h_{0}} &&\mbox{in case (ii)}
        \end{aligned}\right. \\
        &0=\alpha(\iota h,H')_{0}f=\left\{\begin{aligned}&\mu\alpha(\iota h,H')_{0}f_{h}f_{\beta} &&\mbox{in case (i)} \\
        &\mu\alpha(\iota h,H')_{0}f_{h}f_{\beta}+\mu'\alpha(\iota h,H')_{0}f_{h}f_{\beta}f_{h_{0}} &&\mbox{in case (ii)}
        \end{aligned}\right.        
    \end{aligned}
\end{equation*}

\noindent Again, using the correspondence from \cite{OPS}, we know that 

\smallskip

\begin{itemize}[label=\textbullet, font=\tiny]
\item In case (i), $\alpha(h,H')_{0}f_{\beta},\alpha(\iota h,H')_{0}f_{h}f_{\beta}\neq 0$.
\item In case (ii), the morphisms $\alpha(h,H')_{0}f_{\beta}$ and $\alpha(h,H')_{0}f_{\beta}f_{h_{0}}$ are independent and it is also the case for $\alpha(\iota h,H')_{0}f_{h}f_{\beta}$ and $\alpha(\iota h, H')_{0}f_{h}f_{\beta}f_{h_{0}}$.
\end{itemize}

\smallskip

\noindent In both cases, we conclude that $f=0$.

\ref{item2:mutation in gentle algebra} The first part is clear by Proposition \ref{prop:admissible cut}. Moreover, by definition

\[\mu^{+}(\Lambda;e_{H\backslash H'}\Lambda)=e_{H\backslash H'}\Lambda \ \bigoplus \underset{[h]\in H'/\iota}{\bigoplus}Ce_{[h]}\]

\noindent where $Ce_{[h]}$ is the cone of the left $\mathrm{add}(e_{H\backslash H'}\Lambda)$-approximation of $e_{[h]}\Lambda$ in $\per{\Lambda}$. Moreover, by Proposition \ref{prop:OPS} and \cite{APS}, its endomorphism algebra is the gentle algebra corresponding to the marked surface $(S,M)$ equipped with the following admissible dissection

\begin{equation*}
        \Delta'=\{|\overrightarrow{\iota \sigma^{r(h)+1}h}\ \cdot \ \overrightarrow{h} \ \cdot \ \overrightarrow{\sigma^{r(\iota h)+1}\iota h}|\, ,\,  [h]\in H'/\iota\} \ \cup \ \{|\overrightarrow{h'}|\, ,\,  [h']\in (H\backslash H')/\iota\}
\end{equation*}

\noindent It remains to prove that $(S,M,\Delta')$ is the marked surface equipped with an admissible dissection corresponding to $\mu^{+}_{\pi}(\Gamma,d)$. Note first that using Proposition \ref{prop:commutativity of graded generalized Kauer moves}, the successive $\Z$-graded generalized Kauer moves defining $\mu^{+}_{\pi}(\Gamma,d)$ may be done in any order. Iterating Lemma \ref{lem:move of admissible dissection}, we know that $(S,M)$ is the marked surface associated to $\mu^{+}_{\pi}(\Gamma,d)$. Moreover, it suffices to prove that for any $h\in H'$, 

\[\gamma_{h}:=|\overrightarrow{\iota \sigma^{r(h)+1}h}\ \cdot \ \overrightarrow{h} \ \cdot \ \overrightarrow{\sigma^{r(\iota h)+1}\iota h}|\]

\noindent is the \textcolor{green}{$\circ$}-arc in $S$ corresponding to the edge $[h]$ in $\mu^{+}_{\pi}(\Gamma,d)$ under the embedding of the Brauer graph of $\mu^{+}_{\pi}(\Gamma,d)$ in $(S,M)$. There are three cases to consider according to the existence of $r,r'\ge 0$ such that $\sigma^{r}h\notin H'$ and $\sigma^{r'}\iota h \notin H'$. We will only detail the case where such $r,r'\ge 0$ exist, other cases being similar. In this case, there exist $(h_{1},r_{1}),(h_{2},r_{2})\in\mathrm{Sect}(H',\sigma)$ such that $h=\sigma^{j}h_{1}$ and $\iota h=\sigma^{j'}h_{2}$ for some $j=0,\ldots,r_{1}$ and $j'=0,\ldots,r_{2}$. 

\smallskip

\begin{itemize}[label=\textbullet, font=\tiny]
\item If $(h_{1},r_{1})=(h_{2},r_{2})$, by Lemma \ref{lem:move of admissible dissection}, the edge $[h]$ in $\mu^{+}_{(h_{1},r_{1})}(\Gamma,d)$ embeds in $(S,M)$ into the $\textcolor{green}{\circ}$-arc $|\overrightarrow{\iota\sigma^{r_{1}+1}h_{1}}\ \cdot \ \overrightarrow{h} \ \cdot \ \overrightarrow{\sigma^{r_{1}+1}\iota h_{1}}|$ which is exactly $\gamma_{h}$. Moreover, since $[h]\neq[\sigma^{j}h_{0}]$, $j=0,\ldots, r_{0}$ for all $(h_{0},r_{0})\in\mathrm{Sect}(H',\sigma)$ distinct from $(h_{1},r_{1})$, the half edges $h$ and $\iota h$ do not move from $\mu^{+}_{(h_{1},r_{1})}(\Gamma,d)$ to $\mu^{+}_{\pi}(\Gamma,d)$. Thus, the edge $[h]$ in $\mu^{+}_{\pi}(\Gamma,d)$ embeds in $(S,M)$ into the same $\textcolor{green}{\circ}$-arc as the edge $[h]$ in $\mu^{+}_{(h_{1},r_{1})}(\Gamma,d)$ i.e. into $\gamma_{h}$.

\item If $(h_{1},r_{1})\neq (h_{2},r_{2})$, by Lemma \ref{lem:move of admissible dissection}, the oriented edge $(\iota h)^{-1}h$ in $\mu^{+}_{(h_{1},r_{1})}(\Gamma,d)$ embeds in $(S,M)$ into the oriented \textcolor{green}{$\circ$}-arc $(\overrightarrow{\iota\sigma^{r_{1}+1}h_{1}} \ \cdot \ \overrightarrow{h})^{-1}$ i.e. $\overrightarrow{\iota h} \ \cdot \ \overrightarrow{\sigma^{r_{1}+1}h_{1}}$. Thus, using again Lemma \ref{lem:move of admissible dissection}, the edge $[\iota h]$ in $\mu^{+}_{(h_{2},r_{2})(h_{1},r_{1})}(\Gamma,d)$ embeds in $(S,M)$ into the $\textcolor{green}{\circ}$-arc $|\overrightarrow{\iota\sigma^{r_{2}+1}h_{2}}\ \cdot \ \overrightarrow{\iota h}\ \cdot \ \overrightarrow{\sigma^{r_{1}+1}h_{1}}|$ which happens to be $\gamma_{h}$. We conclude as in the previous case.
\end{itemize}

\end{proof}

\medskip

\subsection{Compatibility for Brauer graph algebras}

\medskip

The goal of this part is to prove the following theorem that interprets the generalized Kauer moves in terms of silting mutations.

\medskip

\begin{thm}\label{thm:generalized Kauer move}
    Let $\Gamma=(H,\iota,\sigma)$ be a Brauer graph and $H'$ be a subset of $H$ stable under $\iota$. Denoting $B_{\Gamma}$ and $B_{\mu^{+}_{H'}(\Gamma)}$ the Brauer graph algebras associated to $\Gamma$ and $\mu^{+}_{H'}(\Gamma)$ respectively, 

    \[\mathrm{End}_{\per{B_{\Gamma}}}(\mu^{+}(B_{\Gamma};e_{H\backslash H'}B_{\Gamma}))\simeq B_{\mu^{+}_{H'}(\Gamma)}\]
\end{thm}

\medskip

The idea of the proof is to use a result proved by Rickard in \cite{Rickard} that gives a link between derived equivalences of algebras and derived equivalences of their trivial extensions.

\medskip

\begin{thm}[Rickard {\cite[Theorem 3.1]{Rickard}}]\label{thm:Rickard}
    Let $\Lambda$ be a finite dimensional $k$-algebra and $T\in\per{\Lambda}$ be a tilting object. We denote by $B_{\Lambda}$ the trivial extension of $\Lambda$. Then, $T\lotimes{\Lambda}{B_{\Lambda}}$ is a tilting object in $\per{B_{\Lambda}}$. Moreover, 

    \[\mathrm{End}_{\per{B_{\Lambda}}}(T\lotimes{\Lambda}{B_{\Lambda}})\simeq \mathrm{Triv}(\mathrm{End}_{\per{\Lambda}}(T)) \]

\end{thm}

\medskip

We can summarize the previous theorem with the following commutative diagram

\smallskip

    \begin{center}
        \begin{tikzcd}[column sep=2.8cm, row sep=2.2cm, every label/.append style = {font = \small}]
            \per{\Lambda'} \arrow[cramped]{d}[left]{-\lotimes{\Lambda'}{B_{\Lambda'}}} \arrow[cramped]{r}[above]{-\lotimes{\Lambda'}{T}} &\per{\Lambda} \arrow[cramped]{d}[right]{-\lotimes{\Lambda}{B_{\Lambda}}} \\
            \per{B_{\Lambda'}} 
            \arrow[cramped]{r}[below, yshift=-1mm]{-\lotimes{B_{\Lambda'}}(T\lotimes{\Lambda}B_{\Lambda})} &\per{B_{\Lambda}} 
        \end{tikzcd}
    \end{center}

\smallskip

\noindent where $\Lambda'$ denotes the endomorphism algebra of $T$ and $B_{\Lambda'}$ denotes the trivial extension of $\Lambda'$.

    To prove Theorem \ref{thm:generalized Kauer move}, we will need the following lemma, which links tilting mutations in an algebra and tilting mutations in its trivial extension as in the previous theorem.

    \medskip

    \begin{lem}\label{lem:tilting mutations}
    Let $\Lambda$ be a finite dimensional $k$-algebra and $T\in\per{\Lambda}$ be a tilting object. We denote by $B_{\Lambda}$ the trivial extension of $\Lambda$. Let $T_{0}\in\mathrm{add}(T)$ be such that $\mu^{+}(T;T/T_{0})$ is tilting in $\per{\Lambda}$. Then,

    \[\mu^{+}(T\lotimes{\Lambda}{B_{\Lambda}};(T\lotimes{\Lambda}{B_{\Lambda}})/(T_{0}\lotimes{\Lambda}{B_{\Lambda}}))=\mu^{+}(T;T/T_{0})\lotimes{\Lambda}{B_{\Lambda}}\]
        
    \end{lem}

    \medskip

    In particular, $\mu^{+}(T\lotimes{\Lambda}{B_{\Lambda}};(T\lotimes{\Lambda}{B_{\Lambda}})/(T_{0}\lotimes{\Lambda}{B_{\Lambda}}))\in\per{B_{\Lambda}}$ is tilting using Theorem \ref{thm:Rickard}.

    \medskip

    \begin{proof}
        Let us consider the following triangle in $\per{\Lambda}$

        \smallskip

        \begin{center}
            \begin{tikzcd}
                T_{0} \arrow[cramped]{r}[above]{f} &T' \arrow[cramped]{r} &C_{T_{0}} \arrow[cramped]{r}&T_{0}[1]
            \end{tikzcd}
        \end{center}

        \smallskip

        \noindent where $f:T_{0}\rightarrow T'$ is a left $\mathrm{add}(T/T_{0})$-approximation of $T_{0}$ in $\per{\Lambda}$. Since the functor $-\lotimes{\Lambda}{B_{\Lambda}}:\per{\Lambda}\rightarrow\per{B_{\Lambda}}$ is triangulated, we have

        \[(T\lotimes{\Lambda}{B_{\Lambda}})/(T_{0}\lotimes{\Lambda}{B_{\Lambda}})=(T/T_{0})\lotimes{\Lambda}{B_{\Lambda}}\]

        \noindent Moreover, the previous functor yields the following triangle in $\per{B_{\Lambda}}$

        \smallskip
        
        \begin{center}
            \begin{tikzcd}[column sep=1.5cm]
                T_{0}\lotimes{\Lambda}{B_{\Lambda}} \arrow[cramped]{r}[above, yshift=1mm]{f\lotimes{\Lambda}{B_{\Lambda}}} &T'\lotimes{\Lambda}{B_{\Lambda}} \arrow[cramped]{r} &C_{T_{0}}\lotimes{\Lambda}{B_{\Lambda}} \arrow[cramped]{r}&T_{0}\lotimes{\Lambda}{B_{\Lambda}}[1]
            \end{tikzcd}
        \end{center}

        \smallskip

        \noindent We want to prove that $f\lotimes{\Lambda}{B_{\Lambda}}:T_{0}\lotimes{\Lambda}{B_{\Lambda}}\rightarrow T'\lotimes{\Lambda}{B_{\Lambda}}$ is a left $\mathrm{add}((T/T_{0})\lotimes{\Lambda}{B_{\Lambda}})$-approximation of $T_{0}\lotimes{\Lambda}{B_{\Lambda}}$ in $\per{B_{\Lambda}}$. Applying $\Hom[\per{B_{\Lambda}}]{-}{(T/T_{0})\lotimes{\Lambda}{B_{\Lambda}}}$ to the previous triangle, it suffices to prove that 
        
        \[\mathrm{H}:=\Hom[\per{\Lambda}]{C_{T_{0}}\lotimes{\Lambda}{B_{\Lambda}}[-1]}{(T/T_{0})\lotimes{\Lambda}{B_{\Lambda}}}=0\]
        
        For this, notice that $\mathrm{H}$ is naturally isomorphic to the homology in degree 1 of the total complex of the double complex
        
        \begin{equation*}
        \begin{aligned}
        CC:&=\Hom[B_{\Lambda}]{C_{T_{0}}\lotimes{\Lambda}{B_{\Lambda}}}{(T/T_{0})\lotimes{\Lambda}{B_{\Lambda}}} \\
        &=\Hom[\Lambda]{C_{T_{0}}}{(T/T_{0})\lotimes{\Lambda}{B_{\Lambda}}} \\
        &=\Hom[\Lambda]{C_{T_{0}}}{T/T_{0}}\oplus \Hom[\Lambda]{C_{T_{0}}}{(T/T_{0})\lotimes{\Lambda}{D\Lambda}} \\
        &=\Hom[\Lambda]{C_{T_{0}}}{T/T_{0}}\oplus D\Hom[\Lambda]{T/T_{0}}{C_{T_{0}}}
        \end{aligned}
        \end{equation*}

        \noindent Thus, we obtain that

        \begin{equation*}
            \mathrm{H}= \Hom[\per{\Lambda}]{C_{T_{0}}}{(T/T_{0})[1]} \oplus D\Hom[\per{\Lambda}]{T/T_{0}}{C_{T_{0}}[-1]}
        \end{equation*}

        \noindent Since $\mu^{+}(T;T/T_{0})= T/T_{0} \oplus C_{T_{0}}$ is tilting in $\per{\Lambda}$, we conclude that $H=0$.
    \end{proof}

    \medskip

    We have now all the tools to prove Theorem \ref{thm:generalized Kauer move}.

    \medskip

    \begin{proof}[Proof of Theorem \ref{thm:generalized Kauer move}]
        Let us define an admissible cut $d:H\rightarrow \Z$ on $\Gamma$ as follows : for every vertex $v\in H/\sigma$,

        \begin{itemize}[label=\textbullet, font=\tiny]
            \item If there exists $h_{0}\in H\backslash H'$ and $h_{1}\in H'$ whose source vertices are $v$, then we choose any maximal sector $(h,r)\in\mathrm{Sect}(H',\sigma)$ and we set $d(\sigma^{-1}h)=1$ and $d(h')=0$ for every half-edge $h'\neq \sigma^{-1}h$ whose source vertex is $v$.

            \item Else, we choose any $h\in H$ whose source vertex is $v$ and we set $d(h)=1$ and $d(h')=0$ for every half-edge $h'\neq h$ whose source vertex is $v$.
        \end{itemize}

        \noindent By construction, $(\Gamma,d)$ defines a gentle algebra $\Lambda$ satisfying that $d(\alpha(h,H'))=0$ for all $h\in H'$. In particular, $B_{\Gamma}$ is the trivial extension of $\Lambda$ by Theorem \ref{thm:Schroll}. Moreover, by Theorem \ref{thm:mutation in gentle algebra},  $T:=\mu^{+}(\Lambda;e_{H\backslash H'}\Lambda)$ is tilting and its endomorphism algebra is the gentle algebra associated to $\mu^{+}_{\pi}(\Gamma,d)$, where $\pi$ is the product of all maximal sectors $(h,r)\in\mathrm{Sect}(H',\sigma)$ in $\Gamma$. Thus, by Lemma \ref{lem:tilting mutations}, we obtain the following isomorphism
        
        \begin{equation*}
                E:=\mathrm{End}_{\per{B_{\Gamma}}}(\mu^{+}(B_{\Gamma};e_{H\backslash H'}B_{\Gamma})) 
                \simeq\mathrm{End}_{\per{B_{\Gamma}}}(\mu^{+}(\Lambda;e_{H\backslash H'}\Lambda)\lotimes{\Lambda}{B_{\Gamma}})
        \end{equation*}

        \noindent Using Theorem \ref{thm:Rickard}, we deduce that the endomorphism algebra $E$ is isomorphic to the trivial extension of $\mathrm{End}_{\per{\Lambda}}(T)$. Using again Theorem \ref{thm:Schroll}, we know that $E$ is a Brauer graph algebra whose Brauer graph is $(H,\iota,\sigma_{\pi})$, where $\sigma_{\pi}$ is the orientation of $\mu^{+}_{\pi}(\Gamma,d)$. Thanks to Proposition \ref{prop:commutativity of graded generalized Kauer moves}, notice that the orientation $\sigma_{\pi}$ of $\mu^{+}_{\pi}(\Gamma,d)$ is exactly the orientation $\sigma_{H'}$ of $\mu^{+}_{H'}(\Gamma)$. Thus, $E$ is indeed isomorphic to $B_{\mu^{+}_{H'}(\Gamma)}$. 
    \end{proof}

    \medskip

    Since by Proposition \ref{prop:tilting mutation in symmetric algebra}, $\mu^{+}(B_{\Gamma},e_{H\backslash H'}B_{\Gamma})$ is tilting in $\per{B_{\Gamma}}$, we have another proof of Proposition \ref{prop:derived equivalent Brauer graph algebras} i.e. $B_{\Gamma}$ and $B_{\mu^{+}_{H'}(\Gamma)}$ are derived equivalent Brauer graph algebras.

    \medskip
    
    \begin{ex}
        Let us consider $\Gamma=(H,\iota,\sigma)$ the following Brauer graph

        \smallskip

        
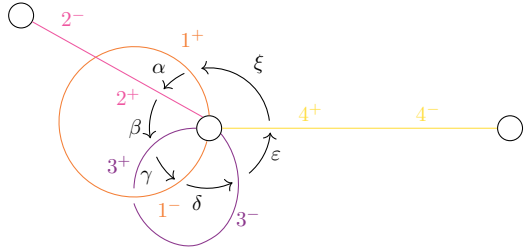
\begin{figure}[H]
            \centering
            \begin{tikzpicture}
        \tikzstyle{vertex}=[draw, circle, minimum size=0.1cm]
        \node[vertex] (1) at (-2,0) {};
        \node[vertex] (2) at (2,0) {};
        \node[vertex] (3) at (-4.5,1.5) {};
        \draw[Goldenrod] (1)--(2) node[near start, above right, scale=0.8]{$4^{+}$} node[near end, above, scale=0.8]{$4^{-}$};
        \draw[Orange] (1.90) arc(5:345:1) node[very near start, above right, scale=0.8, yshift=2mm, xshift=-2mm]{$1^{+}$} node[very near end, below left, yshift=-1mm, xshift=2mm, scale=0.8]{$1^{-}$};
        \draw[VioletRed] (1.120)--(3) node[near start, below left, scale=0.8, yshift=2mm, xshift=-2mm]{$2^{+}$} node[very near end, above right, scale=0.8]{$2^{-}$};
        \draw[Fuchsia] (1.-180) to[bend right=45] node[near end, left, scale=0.8]{$3^{+}$} (-3,-0.8);
        \draw[Fuchsia, rounded corners] (-3,-1) to[bend right=45]  (-2,-1.5) to[bend right=45] node[near start, right, scale=0.8]{$3^{-}$} (1.-25);
        \draw[->] (-1.2,0.1) arc(7:98:0.8) node[midway, above right, scale=0.8]{$\xi$};
        \draw[->] (-2.33,0.73) arc(115:140:0.8) node[midway, above left, scale=0.8]{$\alpha$} ;
        \draw[->] (-2.7,0.38) arc(152:190:0.8) node[near end, left, scale=0.8]{$\beta$};
        \draw[->] (-2.7,-0.38) arc(205:233:0.8) node[near start, below left, scale=0.8]{$\gamma$};
        \draw[->] (-2.3,-0.75) arc(250:295:0.8) node[midway, below left, scale=0.8]{$\delta$} ;
        \draw[->] (-1.47,-0.6) arc(310:355:0.8) node[midway, right, scale=0.8]{$\varepsilon$};
        \end{tikzpicture}
            \label{Final example}
        \end{figure}

        \smallskip

        \noindent where $j^{-}$ denotes $\iota j^{+}$ and $\sigma=(1^{+} \ 2^{+} \ 3^{+} \ 1^{-} \ 3^{-} \ 4^{+})$. Its associated Brauer graph algebra $B_{\Gamma}=kQ/I$ is the path algebra whose quiver is given by

        \smallskip
        
        
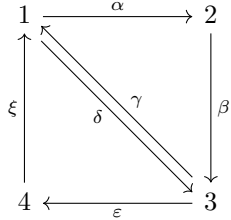
\begin{center}
            \begin{tikzcd}[row sep=2cm, column sep=2cm]
                1 \arrow[cramped]{r}[above]{\alpha} \arrow[cramped, yshift=-1mm]{dr}[xshift=-1mm, left]{\delta} & 2 \arrow[cramped]{d}[right]{\beta} \\
                4 \arrow[cramped]{u}[left]{\xi} & 3 \arrow[cramped, yshift=1mm]{ul}[xshift=1mm, right]{\gamma} \arrow[cramped]{l}[below]{\varepsilon}
            \end{tikzcd}
        \end{center}

        \smallskip

        \noindent and whose set of relations is

        \smallskip
        
        \begin{enumerate}[label=(\Roman*)]
            \item $\xi\varepsilon\delta\gamma\beta\alpha-\gamma\beta\alpha\xi\varepsilon\delta$, $\delta\gamma\beta\alpha\xi\varepsilon-\beta\alpha\xi\varepsilon\delta\gamma$;
            \item $\alpha\xi\varepsilon\delta\gamma\beta\alpha$, $\delta\gamma\beta\alpha\xi\varepsilon\delta$, $\beta\alpha\xi\varepsilon\delta\gamma\beta$, $\varepsilon\delta\gamma\beta\alpha\xi\varepsilon$, $\gamma\beta\alpha\xi\varepsilon\delta\gamma$, $\xi\varepsilon\delta\gamma\beta\alpha\xi$;
            \item $\delta\xi$, $\varepsilon\beta$, $\alpha\gamma$, $\gamma\delta$.
        \end{enumerate}

        \smallskip
        
        \noindent Let us compute the endomorphism algebra of $\mu^{+}(B_{\Gamma};e_{2}B_{\Gamma}\, \oplus \,e_{4}B_{\Gamma})$. By Theorem \ref{thm:generalized Kauer move}, we have to compute the Brauer graph algebra associated to $\mu^{+}_{H'}(\Gamma)$ where $H'=\{1^{+},1^{-},3^{+},3^{-}\}$. The maximal sectors of elements in $H'$ are $(1^{+},0)$ and $(3^{+},2)$. Thus, the Brauer graph obtained from $\Gamma$ by a generalized Kauer move of $H'$ is given by $\mu^{+}_{H'}(\Gamma)=(H,\iota,\sigma_{H'})$ where

        \[\sigma_{H'}=(1^{+} \ 2^{+})(3^{+} \ 4^{+})\sigma(1^{+} \ 2^{-})(3^{-} \ 4^{-})=(1^{+} 2^{-})(1^{-} \ 3^{-} \ 4^{-} \ 3^{+})(2^{+} \ 4^{+})\]

        
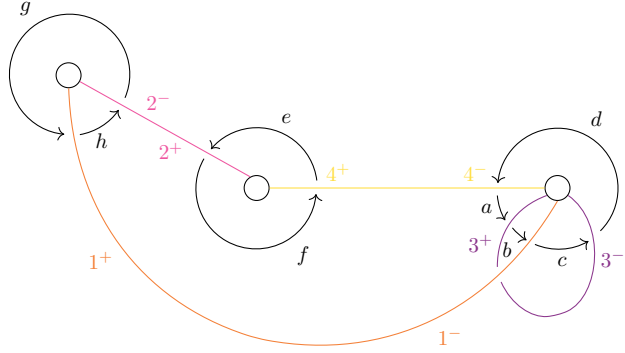
\smallskip
        
        \begin{figure}[H]
            \centering
        \begin{tikzpicture}
        \tikzstyle{vertex}=[draw, circle, minimum size=0.1cm]
        \node[vertex] (1) at (-2,0) {};
        \node[vertex] (2) at (2,0) {};
        \node[vertex] (3) at (-4.5,1.5) {};
        \draw[Goldenrod] (1)--(2) node[near start, above, scale=0.8]{$4^{+}$} node[near end, above, scale=0.8]{$4^{-}$};
        \draw[VioletRed] (1.120)--(3) node[midway, below, scale=0.8, yshift=-1mm, xshift=1mm]{$2^{+}$} node[midway, above, scale=0.8, yshift=2mm, xshift=-1mm]{$2^{-}$};
        \draw[Orange, rounded corners] (3) to[bend right=35] node[midway, below left, scale=0.8]{$1^{+}$} (-2,-2)  to[bend right=35] node[midway, below right, scale=0.8]{$1^{-}$} (2.-90);
        \draw[Fuchsia] (2.-145) to[bend right=35] node[near end, left, scale=0.8]{$3^{+}$} (1.2,-1.05);
        \draw[Fuchsia, rounded corners] (1.25,-1.25) to[bend right=35] (2,-1.7) to[bend right=60] node[midway, right, scale=0.8]{$3^{-}$} (2.-35);
        \draw[->] (-4.35,0.71) arc(-80:-35:0.8) node[midway, below, scale=0.8]{$h$};
        \draw[->] (-3.75, 1.2) arc(-23:265:0.8) node[midway, above left, scale=0.8]{$g$};
        \draw[->] (-1.2, 0.1) arc(7:140:0.8) node[midway, above right, scale=0.8]{$e$};
        \draw[->] (-2.7,0.39) arc(150:355:0.8) node[near end, below right , scale=0.8]{$f$};
        \draw[->] (1.2,-0.1) arc(185:210:0.8) node[midway, left, scale=0.8]{$a$};
        \draw[->] (1.4,-0.53) arc(220:238:0.8) node[midway, below left, scale=0.8]{$b$};
        \draw[->] (1.7,-0.75) arc(247:300:0.8) node[midway, below, scale=0.8]{$c$};
        \draw[->] (2.57, -0.57) arc(315:535:0.8)  node[midway, above right,scale=0.8]{$d$}  ;
        \end{tikzpicture}
        
            \label{Mutation of final example}
        \end{figure}

        \smallskip

        \noindent The Brauer graph algebra $B_{\mu^{+}_{H'}(\Gamma)}=kQ_{H'}/I_{H'}$ associated to $\mu^{+}_{H'}(\Gamma)$ is the path algebra whose quiver is given by

        \smallskip
        
        
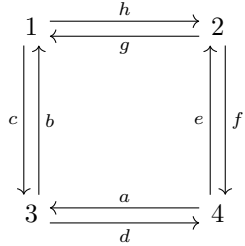
\begin{center}
            \begin{tikzcd}[row sep=2cm, column sep=2cm]
                1 \arrow[cramped, yshift=1mm]{r}[above]{h} \arrow[cramped, xshift=-1mm]{d}[left]{c} &2 \arrow[cramped, yshift=-1mm]{l}[below]{g} \arrow[cramped, xshift=1mm]{d}[right]{f} \\
                3 \arrow[cramped, yshift=-1mm]{r}[below]{d} \arrow[cramped, xshift=1mm]{u}[right]{b} &4 \arrow[cramped, xshift=-1mm]{u}[left]{e} \arrow[cramped, yshift=1mm]{l}[above]{a}
            \end{tikzcd}
        \end{center}

        \smallskip
        \noindent and whose set of relations is 

        \smallskip
        
        \begin{enumerate}[label=(\Roman*)]
            \item $gh-badc$, $hg-ef$, $cbad-adcb$, $dcba-fe$;
            \item $hgh$, $cbadc$, $ghg$, $fef$, $dcbad$, $badcb$, $adcba$, $efe$;
            \item $fh$, $cg$, $ge$, $af$, $da$, $hb$, $bc$, $ed$.
        \end{enumerate}

        \smallskip

        \noindent In particular, the Brauer graph algebras $B_{\Gamma}$ and $B_{\mu^{+}_{H'}(\Gamma)}$ are derived equivalent.

    \end{ex}

\medskip
\phantomsection

\setlength{\bibitemsep}{0pt}
\setlength{\biblabelsep}{8pt}
\defbibnote{nom}{\\ \small{\scshape{Université Grenoble Alpes, CNRS, Institut Fourier, 38610 Gières}} \\ \textit{E-mail address :} \href{mailto:valentine.soto@univ-grenoble-alpes.fr}{valentine.soto@univ-grenoble-alpes.fr}}
\printbibliography[postnote=nom, heading=bibintoc,title={References}]

\end{document}